\documentclass[a4paper,12pt]{amsart}
\usepackage{amsfonts}
\usepackage{amsmath}
\usepackage{hyperref}
\usepackage{amssymb}
\usepackage{color,tikz}

\allowdisplaybreaks
\numberwithin{equation}{section}
\newtheorem{theorem}{Theorem}[section]
\newtheorem{proposition}[theorem]{Proposition}
\newtheorem{lemma}[theorem]{Lemma}
\newtheorem{corollary}[theorem]{Corollary}

\newtheorem{claim}[theorem]{Claim}
\theoremstyle{definition}
\newtheorem{definition}[theorem]{Definition}

\theoremstyle{remark}
\newtheorem{remark}[theorem]{Remark}
\newtheorem{example}[theorem]{Example}

\begin{document}
\title{Kumjian-Pask algebras of finitely-aligned higher-rank graphs}
\author{Lisa Orloff Clark and Yosafat E. P. Pangalela}
\address{Lisa Orloff Clark and Yosafat E. P. Pangalela\\
Department of Mathematics and Statistics\\
University of Otago\\
PO Box 56\\
Dunedin 9054\\
New Zealand}
\email{lclark@maths.otago.ac.nz, yosafat.pangalela@maths.otago.ac.nz}

\begin{abstract}
We extend the the definition of Kumjian-Pask algebras to include algebras
associated to finitely aligned higher-rank graphs. We show that these
Kumjian-Pask algebras are universally defined and have a graded uniqueness
theorem. We also prove the Cuntz-Kreiger uniqueness theorem; to do this, we
use a groupoid approach. As a consequence of the graded uniqueness theorem,
we show that every Kumjian-Pask algebra is isomorphic to the Steinberg
algebra associated to its boundary path groupoid. We then use Steinberg
algebra results to prove the Cuntz-Kreiger uniqueness theorem and also to
characterize simplicity and basic simplicity.
\end{abstract}

\thanks{This research was done as part of the second author's PhD thesis at
the University of Otago under the supervision of the first author and Iain
Raeburn. Thank you to Iain for his guidance.}
\maketitle

\section{Introduction}

In the 1990s, $C^{\ast }$-algebras of row-finite directed graphs were
introduced in \cite{BPRS00,KPR98,KPRR97}. Since their first appearance,
these $C^{\ast }$-algebras have been intensively studied (for example, see
\cite{R08}). Some of the earliest results about these algebras include the
existence of a universal family, the gauge-invariant uniqueness theorem, and
the Cuntz-Krieger uniqueness theorem. 

Higher-rank graph $C^{\ast}$-algebras were introduced by Kumjian and Pask in
\cite{KP00} as a generalisation of the $C^{\ast }$-algebras of directed
graphs. In \cite{KP00}, Kumjian and Pask limit their focus to row-finite
higher-rank graphs with no sources. Later, Raeburn, Sims and Yeend extended
the coverage by introducing $C^{\ast}$-algebras of locally convex,
row-finite higher-rank graphs in \cite{RSY03} and then finitely aligned
higher-rank graphs in \cite{RSY04}. It is in the finitely aligned setting
where graphs that fail to be row-finite are considered. Once again Raeburn,
Sims and Yeend establish the existence of a universal family, the
gauge-invariant uniqueness theorem, and the Cuntz-Krieger uniqueness theorem.

On the other hand, Leavitt path algebras were developed independently by
Ara, Moreno, and Pardo in \cite{AMP07} and Abrams and Aranda Pino in \cite%
{AA05}. A complex Leavitt path algebra is a purely algebraic structure
constructed from a directed graph that sits densely inside the graph $%
C^{\ast}$-algebra. Tomforde showed in \cite{T11} that one can generalise
further and define Leavitt path $R$-algebras where $R$ is any commutative
ring with identity. Tomforde proved the existence of a universal family, the
graded uniqueness theorem (which is the algebraic analogue of the
gauge-invariant uniqueness theorem), and the Cuntz-Krieger uniqueness
theorem for Leavitt path $R$-algebras. Tomforde's proofs in \cite{T11} use
techniques that are similar to those employed by Raeburn for Leavitt path $%
\mathbb{C}$-algebras in \cite{Graph Algebras} and in Tomforde's earlier
paper \cite{T07} for Leavitt path $K$-algebras where $K $ is an arbitrary
field.

Moving to higher-rank graphs, Kumjian-Pask $R$-algebras were introduced in
\cite{ACaHR13} and include the class of Leavitt path algebras. Kumjian-Pask
algebras are the algebraic analogue of the higher-rank graph $C^{\ast}$%
-algebras of \cite{KP00}. As in \cite{KP00}, the authors of \cite{ACaHR13}
consider row-finite higher-rank graphs with no sources. Later, Clark, Flynn
and an Huef developed Kumjian-Pask algebras for locally convex, row-finite
higher-rank graphs in \cite{CFaH14}. To complete the final algebraic piece,
in this paper we introduce Kumjian-Pask algebras for finitely aligned
higher-rank graphs. We will establish the existence of a universal family,
the graded-invariant uniqueness theorem, and the Cuntz-Krieger uniqueness
theorem.

Our motivation to consider this class of higher-rank graphs comes from our
desire to establish an algebraic version of \cite[Theorem~4.1]{P15}: there
Pangalela shows that the Toeplitz $C^{\ast}$ algebra associated to a
row-finite graph $\Lambda$ can be realized as the graph $C^{\ast}$-algebra
associated to a higher-rank graph constructed from $\Lambda$, called $%
T\Lambda$. In this setting $T\Lambda$ has sources and is not locally convex.

Let $\Lambda $ be a finitely aligned $k$-graph and let $R$ be a commutative
ring with identity. We define a Kumjian-Pask $\Lambda $-family (Definition %
\ref{KP-family}) and show the existence of a universal Kumjian-Pask algebra $%
{\normalsize \operatorname{KP}}_{R}\left( \Lambda \right) $ that is a $\mathbb{Z}%
^{k} $-graded $R$-algebra in Proposition \ref{universal-KP-family}. We then
prove the graded-invariant uniqueness theorem in Theorem \ref%
{the-graded-uniqueness-theorem}. Up to this point, our techniques mirror the
$C^{\ast }$-algebraic techniques of \cite{RSY04}. However, the proof of the
Cuntz-Krieger uniqueness theorem of \cite{RSY04} is highly analytic so we
must use an alternate approach. We have chosen a groupoid approach.

In Section \ref{Section-Steinberg-algebra}, we introduce groupoids and \emph{%
Steinberg algebras}. Then, given a finitely aligned higher-rank graph $%
\Lambda $, we build the associated boundary-path groupoid $\mathcal{G}%
_{\Lambda }$ as in \cite{Y07}. We then use the graded-invariant uniqueness
theorem (Theorem \ref{the-graded-uniqueness-theorem}) to show that the
Kumjian-Pask algebra ${\normalsize \operatorname{KP}}_{R}\left( \Lambda \right) $ is
isomorphic to the Steinberg algebra $A_{R}(G_{\Lambda })$ in Proposition \ref%
{KP-is-isomorphic-to-Steinberg-algebras}. With this isomorphism in place, we
aim to use results about Steinberg algebras to establish results about
Kumjian-Pask algebras.

First we establish how certain properties of $\Lambda$ translate to
properties of $\mathcal{G}_{\Lambda}$; we do this in Section \ref%
{Section-aperiodic-effective} and Section \ref{Section-cofinal-minimal}. Of
interest in its own right, we show that a higher-rank graph $\Lambda $ is
\emph{aperiodic} if and only if the boundary-path groupoid $\mathcal{G}%
_{\Lambda }$ is \emph{effective} in Proposition \ref{aperiodic-iff-effective}%
. We also show in Proposition \ref{cofinal-iff-minimal}, that a higher-rank
graph $\Lambda $ is \emph{cofinal} if and only if $\mathcal{G}_{\Lambda }$
is \emph{minimal}.

Now in Section \ref{Section-the-CK-uniqueness-theorem}, we prove the
Cuntz-Krieger uniqueness theorem. This theorem only applies to Kumjian-Pask
algebras associated to aperiodic graphs. The proof is a simpy an application
of the Cuntz-Krieger uniqueness theorem for Steinberg algebras \cite[Theorem
3.2]{CE-M15} which applies to effective groupoids. Note that our technique
gives an alternate proof of the Cuntz-Krieger uniqueness theorem in the
special cases of Leavitt path algebras in \cite{T11} and the row-finite
Kumjian-Pask algebras of \cite{ACaHR13,CFaH14}.

Finally, in Section \ref{Section-Basic-Simpllicity}, we give necessary and
sufficient conditions for ${\normalsize \operatorname{KP}}_{R}\left( \Lambda \right)$
to be basically simple in Theorem \ref{basic-simplicity} and simple in
Theorem \ref{simplicity}. These two results are a consequence of the
characterisation of basic simplicity and simplicity of the Steinberg algebra
$A_{R}\left( \mathcal{G}_{\Lambda }\right) $ (see Theorem 4.1 and Corollary
4.6 of \cite{CE-M15}).

\section{Background}

Let $\mathbb{N}$ be the set of non-negative integers and let $k$ be a
positive integer. We write $n\in \mathbb{N}^{k}$ as $\left( n_{1},\ldots
,n_{k}\right) $ and for $m,n\in \mathbb{N}^{k}$, we write $m\leq n$ to
denote $m_{i}\leq n_{i}$ for $1\leq i\leq k$. We also write $m\vee n$ for
their coordinate-wise maximum and $m\wedge n$ for their coordinate-wise
minimum. We denote the usual basis in $\mathbb{N}^{k}$ by $\left\{
e_{i}\right\} $.

A \emph{directed graph} or $1$\emph{-graph} $E=\left( E^{0},E^{1},r,s\right)
$ consists of countable sets of vertices $E^{0}$, edges $E^{1}$ and
functions $r,s:E^{1}\rightarrow E^{0}$, which denote range and source maps,
respectively. We follow the conventions of \cite{CBMS} and write $\lambda
\mu $ to denote the composition of paths $\lambda $ and $\mu $ with $s\left(
\lambda \right) =r\left( \mu \right) $. Thus a path of length $n\in \mathbb{N%
}$ is a sequence $\lambda =\lambda _{1}\cdots \lambda _{n}$ of edges $%
\lambda _{i}$ with $s\left( \lambda _{i}\right) =r\left( \lambda
_{i+1}\right) $ for $1\leq i\leq n-1$. We also have $s\left( \lambda \right)
:=s\left( \lambda _{n}\right) $ and $r\left( \lambda \right) :=r\left(
\lambda _{1}\right) $.

\begin{remark}
\label{convetion-of-graph}We use this convention of paths because we view
the collection of paths as a category.
\end{remark}

\subsection{Higher-rank graphs.}

For a positive integer $k$, we regard the additive semigroup $\mathbb{N}^{k}$
as a category with one object. A \emph{higher-rank graph} or $k$\emph{-graph}
$\Lambda =\left( \Lambda ^{0},\Lambda ,r,s\right) $ is a countable small
category $\Lambda $ with a functor $d:\Lambda \rightarrow \mathbb{N}^{k}$,
called the \emph{degree map}, satisfying the \emph{factorisation property}:
for every $\lambda \in \Lambda $ and $m,n\in \mathbb{N}^{k}$ with $d\left(
\lambda \right) =m+n$, there are unique elements $\mu ,\nu \in \Lambda $
such that $\lambda =\mu \upsilon $ and $d\left( \mu \right) =m$, $d\left(
\nu \right) =n$. We then write $\lambda \left( 0,m\right) $ for $\mu $ and $%
\lambda \left( m,m+n\right) $ for $\nu $.

We write $\Lambda ^{0}$ to denote the set of objects in $\Lambda $ and we
identify each object $v\in \Lambda ^{0}$ with the identity morphism at the
object, which, by the factorisation property, is the only morphism with
range and source $v$. We then regard elements of $\Lambda ^{0}$ as \emph{%
vertices}$.$ For $n\in \mathbb{N}^{k}$, we define%
\begin{equation*}
\Lambda ^{n}:=\left\{ \lambda \in \Lambda :d\left( \lambda \right) =n\right\}
\end{equation*}%
and call the elements $\lambda $ of $\Lambda ^{n}$ \emph{paths of degree }$n$%
. For each $\lambda \in \Lambda $ we say $\lambda $ has \emph{\ source }$%
s\left( \lambda \right) $ and \emph{\ range } $r\left( \lambda \right) $.
For $v\in \Lambda ^{0}$, $\lambda \in \Lambda $ and $E\subseteq \Lambda $,
we define
\begin{align*}
vE& :=\left\{ \mu \in E:r\left( \mu \right) =v\right\} \text{,} \\
\lambda E& :=\left\{ \lambda \mu \in \Lambda :\mu \in E,r\left( \mu \right)
=s\left( \lambda \right) \right\} \text{,} \\
E\lambda & :=\left\{ \mu \lambda \in \Lambda :\mu \in E,s\left( \mu \right)
=r\left( \lambda \right) \right\} \text{.}
\end{align*}

\begin{remark}
In older references, for example \cite{KP00,RSY03}, $v\Lambda $ is denoted
by $\Lambda \left( v\right) $.
\end{remark}

\begin{example}[{\protect\cite[Example 2.2.(ii)]{RSY03}}]
Let $k\in \mathbb{N}$ and $m\in \left( \mathbb{N\cup }\left\{ \infty
\right\} \right) ^{k}$. We define%
\begin{equation*}
\Omega _{k,m}:=\left\{ \left( p,q\right) \in \mathbb{N}^{k}\times \mathbb{N}%
^{k}:p\leq q\leq m\right\} \text{.}
\end{equation*}%
This is a category with objects $\left\{ p\in \mathbb{N}^{k}:p\leq m\right\}
$, range map $r\left( p,q\right) =p$, source map $s\left( p,q\right) =q$,
and degree map $d\left( p,q\right) =q-p$. Then $\left( \Omega
_{k,m},d\right) $ is a $k$-graph.
\end{example}

One way to visualise $k$-graphs is to use coloured graphs. By choosing $k$%
-different colours $c_{1},\ldots ,c_{k}$, we can view paths in $\Lambda
^{e_{i}}$ as edges of colour $c_{i}$. For a $k$-graph $\Lambda $, we call
its corresponding coloured graph the \emph{skeleton }of $\Lambda $. For
further discussion about $k$-graphs and their skeletons, see \cite{HRSW13}.

Let $\Lambda $ be a $k$-graph. For $\lambda ,\mu \in \Lambda $, we say that $%
\tau $ is a \emph{minimal common extension }of $\lambda $ and $\mu $ if
\begin{equation*}
d\left( \tau \right) =d\left( \lambda \right) \vee d\left( \mu \right) \text{%
, }\tau \left( 0,d\left( \lambda \right) \right) =\lambda \text{ and }\tau
\left( 0,d\left( \mu \right) \right) =\mu \text{.}
\end{equation*}%
Let $\operatorname{MCE}\left( \lambda ,\mu \right) $ denote the collection of all
minimal common extensions of $\lambda $ and $\mu $. Then we write%
\begin{equation*}
\Lambda ^{\min }\left( \lambda ,\mu \right) :=\left\{ \left( \rho ,\tau
\right) \in \Lambda \times \Lambda :\lambda \rho =\mu \tau \in \operatorname{MCE}%
\left( \lambda ,\mu \right) \right\} \text{.}
\end{equation*}%
Meanwhile, for $E\subseteq \Lambda $ and $\lambda \in \Lambda $, we write
\begin{equation*}
\operatorname{Ext}\left( \lambda ;E\right) :=\bigcup_{\mu \in E}\left\{ \rho :(\rho
,\tau )\in \Lambda ^{\min }\left( \lambda ,\mu \right) \right\} \text{.}
\end{equation*}%
A set $E\subseteq v\Lambda $ is \emph{exhaustive} if for every $\lambda \in
v\Lambda $, there exists $\mu \in E$ such that $\Lambda ^{\min }\left(
\lambda ,\mu \right) \neq \emptyset $. We define%
\begin{equation*}
\operatorname{FE}\left( \Lambda \right) :=\bigcup_{v\in \Lambda ^{0}}\left\{
E\subseteq \left. v\Lambda \right\backslash \left\{ v\right\} :E\text{ is
finite and exhaustive}\right\} \text{.}
\end{equation*}%
For $E\in \operatorname{FE}\left( \Lambda \right) $, we write $r\left( E\right) $
for the vertex $v$ which satisfies $E\subseteq v\Lambda $.

We say that $\Lambda $ is \emph{finitely aligned} if $\Lambda ^{\min }\left(
\lambda ,\mu \right) $ is finite (possibly empty) for all $\lambda ,\mu \in
\Lambda $. We see that every $1$-graph is finitely aligned. As in \cite[%
Definition 1.4]{KP00}, we say that a $k$-graph $\Lambda $ is \emph{row-finite%
} if $v\Lambda ^{n}$ is finite for every $v\in \Lambda ^{0}$ and $n\in
\mathbb{N}^{k}$. Note that for all $\lambda ,\mu \in \Lambda $, we have $%
\left\vert \Lambda ^{\min }\left( \lambda ,\mu \right) \right\vert
=\left\vert \operatorname{MCE}\left( \lambda ,\mu \right) \right\vert \leq
\left\vert r\left( \lambda \right) \Lambda ^{d\left( \lambda \right) \vee
d\left( \mu \right) }\right\vert $. Hence, every row-finite $k$-graph $%
\Lambda $ is finitely aligned. On the other hand, a finitely aligned $k$%
-graph $\Lambda $ is not necessarily row-finite.

For example, consider the $2$-graph $\Lambda _{1}$ which has skeleton
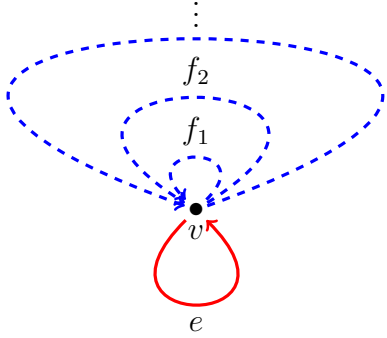
\begin{equation*}
\begin{tikzpicture} \node[inner sep=1pt] (v) at (0,0) {$\bullet$};
\node[inner sep=1pt] at (0,-0.3) {$v$}; \path[->,every
loop/.style={looseness=20}] (v) edge[in=-45,out=-135,loop, red, very thick]
node[pos=0.5, below, black]{$e$} (v); \path[->,every
loop/.style={looseness=10}] (v) edge[in=-225,out=-315,loop, blue, dashed,
very thick] node[pos=0.5, above, black]{$f_1$} (v); \path[->,every
loop/.style={looseness=30}] (v) edge[in=-212.5,out=-327.5,loop, blue,
dashed, very thick] node[pos=0.5, above, black]{$f_2$} (v); \path[->,every
loop/.style={looseness=75}] (v) edge[in=-200,out=-340,loop, blue, dashed,
very thick] node[pos=0.5, above, black]{$\vdots$} (v); \end{tikzpicture}
\end{equation*}
where $ef_{i}=f_{i}e$ for all positive integers $i$, the solid edge has
degree $\left( 1,0\right) $ and dashed edges have degree $\left( 0,1\right) $%
. It is clearly not row-finite because $|v\Lambda _{1}^{\left( 0,1\right)
}|=\infty $. On the other hand, for $\lambda ,\mu \in \Lambda $, $\left\vert
\Lambda _{1}^{\min }\left( \lambda ,\mu \right) \right\vert $ is either $0$
or $1$, and then $\Lambda _{1}$ is finitely aligned.

Following \cite[Definition 1.4]{KP00}, a $k$-graph $\Lambda $ has \emph{no
sources }if $v\Lambda ^{n}$ is nonempty for every $v\in \Lambda ^{0}$ and $%
n\in \mathbb{N}^{k}$. Meanwhile, recall from \cite[Definition 3.9]{RSY03}
that a $k$-graph $\Lambda $ is \emph{locally convex }if for all $v\Lambda
^{0}$, $1\leq i,j\leq k$ with $i\neq j$, $\lambda \in v\Lambda ^{e_{i}}$ and
$\mu \in v\Lambda ^{e_{j}}$, the sets $s\left( \lambda \right) \Lambda
^{e_{j}}$ and $s\left( \mu \right) \Lambda ^{e_{i}}$ are nonempty.

Consider the $2$-graph $\Lambda _{2}$ with skeleton%
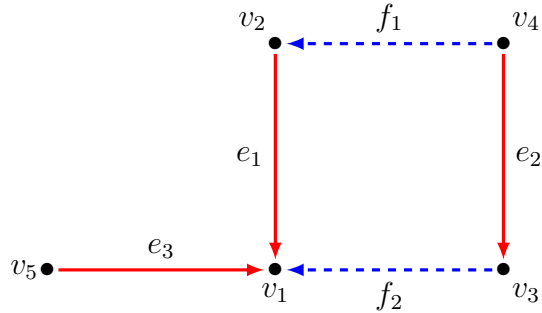
\begin{equation*}
\begin{tikzpicture} \node[inner sep=1pt] (v_1) at (0,0) {$\bullet$};
\node[inner sep=1pt] at (0,-0.3) {$v_1$}; \node[inner sep=1pt] (v_2) at
(0,3) {$\bullet$}; \node[inner sep=1pt] at (-0.3,3.3) {$v_2$}; \node[inner
sep=1pt] (v_3) at (3,0) {$\bullet$}; \node[inner sep=1pt] at (3.3,-0.3)
{$v_3$}; \node[inner sep=1pt] (v_4) at (3,3) {$\bullet$}; \node[inner
sep=1pt] at (3.3,3.3) {$v_4$}; \node[inner sep=1pt] (v_5) at (-3,0)
{$\bullet$}; \node[inner sep=1pt] at (-3.3,0) {$v_5$}; \draw[-latex, red,
very thick] (v_2) edge[out=270, in=90] node[pos=0.5, left, black]{$e_1$}
(v_1); \draw[-latex, blue, dashed, very thick] (v_4) edge[out=180, in=0]
node[pos=0.5, above, black]{$f_1$} (v_2); \draw[-latex, red, very thick]
(v_4) edge[out=270, in=90] node[pos=0.5, right, black]{$e_2$} (v_3);
\draw[-latex, blue, dashed, very thick] (v_3) edge[out=180, in=0]
node[pos=0.5, below, black]{$f_2$} (v_1); \draw[-latex, red, very thick]
(v_5) edge[out=0, in=180] node[pos=0.5, above, black]{$e_3$} (v_1);
\end{tikzpicture}
\end{equation*}
where $e_{1}f_{1}=f_{2}e_{2}$, solid edges have degree $\left( 1,0\right) $
and dashed edges have degree $\left( 0,1\right) $. Since $v_{5}$ does not
receive edges with degree $\left( 0,1\right) $, then $v_{5}$ is a source of $%
\Lambda _{2}$. Furthermore, $\Lambda _{2}$ fails to be locally-convex since $%
e_{3}\in v_{1}\Lambda _{2}^{\left( 1,0\right) },f_{2}\in v_{1}\Lambda
_{2}^{\left( 0,1\right) }$ but $s\left( e_{3}\right) \Lambda _{2}^{\left(
0,1\right) }=\emptyset $. On the other hand, $\Lambda _{2}$ is row-finite
thus $\Lambda _{2}$ is finitely aligned.

Next consider the $2$-graph $\Lambda _{3}$ with skeleton%
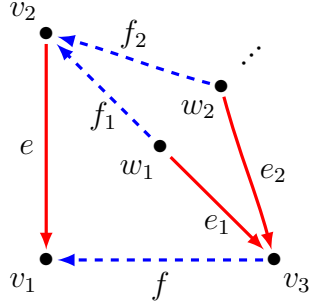
\begin{equation*}
\begin{tikzpicture} \node[inner sep=1pt] (v_1) at (0,0) {$\bullet$};
\node[inner sep=1pt] at (-0.3,-0.3) {$v_1$}; \node[inner sep=1pt] (v_2) at
(0,3) {$\bullet$}; \node[inner sep=1pt] at (-0.3,3.3) {$v_2$}; \node[inner
sep=1pt] (v_3) at (3,0) {$\bullet$}; \node[inner sep=1pt] at (3.3,-0.3)
{$v_3$}; \node[inner sep=1pt] (w_1) at (1.5,1.5) {$\bullet$}; \node[inner
sep=1pt] at (1.2,1.2) {$w_1$}; \node[inner sep=1pt] (w_2) at (2.3,2.3)
{$\bullet$}; \node[inner sep=1pt] at (2.0,2.0) {$w_2$}; \node[inner sep=1pt]
at (2.6,2.6) {$\cdot$}; \node[inner sep=1pt] at (2.7,2.7) {$\cdot$};
\node[inner sep=1pt] at (2.8,2.8) {$\cdot$}; \draw[-latex, red, very thick]
(v_2) edge[out=270, in=90] node[pos=0.5, left, black]{$e$} (v_1);
\draw[-latex, blue, dashed, very thick] (v_3) edge[out=180, in=0]
node[pos=0.5, below, black]{$f$} (v_1); \draw[-latex, red, very thick] (w_1)
edge[out=315, in=135] node[pos=0.5, below, black]{$e_1$} (v_3);
\draw[-latex, blue, dashed, very thick] (w_1) edge[out=135, in=315]
node[pos=0.5, below, black]{$f_1$} (v_2); \draw[-latex, blue, dashed, very
thick] (w_2) edge[out=165, in=-15] node[pos=0.5, above, black]{$f_2$} (v_2);
\draw[-latex, red, very thick] (w_2) edge[out=285, in=105] node[pos=0.5,
right, black]{$e_2$} (v_3); \end{tikzpicture}
\end{equation*}
where $ef_{i}=fe_{i}$ for all positive integers $i$, solid edges have degree
$\left( 1,0\right) $ and dashed edges have degree $\left( 0,1\right) $.
Since $\left\vert \Lambda _{3}^{\min }\left( e,f\right) \right\vert =\infty $%
, then $\Lambda _{3}$ is not finitely aligned. Hence, not every $k$-graph is
finitely aligned.

To summarise, finitely aligned $k$-graphs generalise both row-finite $k$%
-graphs with no sources and locally convex row-finite $k$-graphs. However,
this class of $k$-graphs does not cover all $k$-graphs. In this paper, we
focus on finitely aligned $k$-graphs. For other examples and further
discussion, see \cite{KP00,P15,RSY03,RSY04,W11}.

\subsection{Paths and boundary paths.}

Suppose that $\Lambda $ is a finitely aligned $k$-graph. Recall from \cite[%
Definition 3.1]{RSY03} that for $n\in \mathbb{N}^{k}$, we define%
\begin{equation*}
\Lambda ^{\leq n}:=\{\lambda \in \Lambda :d\left( \lambda \right) \leq n%
\text{, and }d\left( \lambda \right) _{i}<n_{i}\text{ implies }s\left(
\lambda \right) \Lambda ^{e_{i}}=\emptyset \}\text{.}
\end{equation*}%
Note that $v\Lambda ^{\leq n}\neq \emptyset $ for all $v\in \Lambda ^{0}$
and $n\in \mathbb{N}^{k}$. This is because $v$ is contained in $v\Lambda
^{\leq n}$ whenever $v\Lambda ^{\leq n}$ has no non-trivial paths of degree
less than or equal to $q$. For further discussion, see \cite[Remark 3.2]%
{RSY03}.

Following \cite[Definition 5.10]{FMY05}, we say that a degree-preserving
functor $x:\Omega _{k,m}\rightarrow \Lambda $ is a \emph{boundary path }of $%
\Lambda $ if for every $n\in \mathbb{N}^{k}$ with $n\leq m $ and for $E\in
x\left( n,n\right) \operatorname{FE}\left( \Lambda \right) $, there exists $\lambda
\in E$ such that $x\left( n,n+d\left( \lambda \right) \right) =\lambda $. We
write $\partial \Lambda $ for the set of all boundary paths. Note that for $%
v\in \Lambda ^{0}$, $v\partial \Lambda $ is nonempty \cite[Lemma 5.15]{FMY05}%
.

\begin{remark}
In the locally convex setting, the set $\Lambda ^{\leq \infty }$ (as defined
in \cite[Definition 3.14]{RSY03}) is referred to as the \textquotedblleft
boundary path space\textquotedblright . Indeed, if $\Lambda $ is row-finite
and locally convex, then $\Lambda ^{\leq \infty }=\partial \Lambda $ \cite[%
Proposition 2.12]{W11}. However, more generally, $\Lambda ^{\leq \infty }
\subseteq \partial \Lambda $ and the two can be different (see \cite[Example
2.11]{W11}).
\end{remark}

Let $x\in \partial \Lambda $. If $n\in \mathbb{N}^{k}$ and $n\leq d\left(
x\right) $, we define $\sigma ^{n}x$ by $\sigma ^{n}x\left( 0,m\right)
=x\left( n,n+m\right) $ for all $m\leq d\left( x\right) -n$, and by \cite[%
Lemma 5.13.(1)]{FMY05}, $\sigma ^{n}x$ also belongs to $\partial \Lambda $.
We also write $x\left( n\right) $ for the vertex $x\left( n,n\right) $. Then
the range of boundary path $x$ is the vertex $r\left( x\right) :=x\left(
0\right) $. For $\lambda \in \Lambda x\left( 0\right) $, we also have $%
\lambda x\in \partial \Lambda $ \cite[Lemma 5.13.(2)]{FMY05}.

\subsection{Graded rings.}

Suppose that $G$ is an additive abelian group. A ring $A$ is $G$\emph{%
-graded }if there are additive subgroups $\left\{ A_{g}:g\in G\right\} $
satisfying:%
\begin{equation*}
A=\bigoplus {}_{g\in G}A_{g}\text{ and for }g,h\in G\text{, }%
A_{g}A_{h}\subseteq A_{g+h}\text{.}
\end{equation*}%
If $A$ and $B$ are $G$-graded rings, a homomorphism $\pi :A\rightarrow B$ is
$G$\emph{-graded}\ if $\pi \left( A_{g}\right) \subseteq B_{g}$ for $g\in G$.

Let $A$ be a $G$-graded ring. We say an ideal $I$ of $A$ is a $G$\emph{%
-graded ideal} if $\left\{ I\cap A_{g}:g\in G\right\} $ is a grading of $I$.

\section{Kumjian-Pask $\Lambda $-families}

Suppose that $\Lambda $ is a finitely aligned $k$-graph and $R$ is a
commutative ring with identity $1$. For $\lambda \in \Lambda $, we call $%
\lambda ^{\ast }$ a \emph{ghost path }($\lambda ^{\ast }$ is a formal
symbol) and we define
\begin{equation*}
G\left( \Lambda \right) :=\left\{ \lambda ^{\ast }:\lambda \in \Lambda
\right\} \text{.}
\end{equation*}%
For $v\in \Lambda ^{0}$, we define $v^{\ast }:=v$. We also extend $r$ and $s$
to be defined on $G\left( \Lambda \right) $ by%
\begin{equation*}
r\left( \lambda ^{\ast }\right) =s\left( \lambda \right) \text{ and }s\left(
\lambda ^{\ast }\right) =r\left( \lambda \right) \text{.}
\end{equation*}%
We then define composition on $G\left( \Lambda \right) $ by setting $\lambda
^{\ast }\mu ^{\ast }=\left( \mu \lambda \right) ^{\ast }$ for $\lambda ,\mu
\in \Lambda $; and write $G\left( \Lambda ^{\neq 0}\right) $ the set of
ghost paths that are not vertices. Note that the factorisation property of $%
\Lambda $ induces a similar factorisation property on $G\left( \Lambda
\right) $.

\begin{definition}
\label{KP-family}A \emph{Kumjian-Pask }$\Lambda $\emph{-family} $\left\{
S_{\lambda },S_{\mu ^{\ast }}:\lambda ,u\in \Lambda \right\} $ in an $R$%
-algebra $A$ consists of $S:\Lambda \cup G\left( \Lambda ^{\neq 0}\right)
\rightarrow A$ such that:

\begin{enumerate}
\item[(KP1)] $\left\{ S_{v}:v\in \Lambda ^{0}\right\} $ is a collection of
mutually orthogonal idempotents;

\item[(KP2)] for $\lambda ,\mu \in \Lambda $ with $s\left( \lambda \right)
=r\left( \mu \right) $, we have $S_{\lambda }S_{\mu }=S_{\lambda \mu }$ and $%
S_{\mu ^{\ast }}S_{\lambda ^{\ast }}=S_{\left( \lambda \mu \right) ^{\ast }}$%
;

\item[(KP3)] $S_{\lambda ^{\ast }}S_{\mu }=\sum_{(\rho ,\tau )\in \Lambda
^{\min }\left( \lambda ,\mu \right) }S_{\rho }S_{\tau ^{\ast }}$ for all $%
\lambda ,\mu \in \Lambda $; and

\item[(KP4)] $\prod_{\lambda \in E}\left( S_{r\left( E\right) }-S_{\lambda
}S_{\lambda ^{\ast }}\right) =0$ for all $E\in \operatorname{FE}\left( \Lambda
\right) $.
\end{enumerate}
\end{definition}

\begin{remark}
\label{KP-family-remark}A number of aspects of these relations are worth
commenting on:

\begin{enumerate}
\item[(i)] In previous references about Leavitt path algebras and
Kumjian-Pask algebras, people usually distinguish the vertex idempotents as
\textquotedblleft $P_{v}$\textquotedblright\ (for example, see \cite%
{A15,AA05,AA08,AMP07,ACaHR13,CFaH14,T07,T11}). We do not follow this
convention because we do not want to make additional unnecessary cases in
each proof.

\item[(ii)] (KP2) in \cite{ACaHR13,CFaH14} has more relations to check.
However, using our notational convention, those relations can be simplified
and are equivalent to our (KP2).

\item[(iii)] The restriction to finitely aligned $k$-graphs is necessary for
the sum in (KP3) to be make sense (see \cite{RS05}).

\item[(iv)] In (KP3), we interpret the empty sum as $0$, so $S_{\lambda
^{\ast }}S_{\mu }=0$ whenever $\Lambda ^{\min }\left( \lambda ,\mu \right)
=\emptyset $. We also have $S_{\lambda ^{\ast }}S_{\lambda }=S_{s\left(
\lambda \right) }$.

\item[(v)] (KP3-4) have been changed from those in \cite[Definition 3.1]%
{ACaHR13} and \cite[Definition 3.1]{CFaH14}. We do this because we need to
adjust the relations to deal with situation where $k$-graph is not locally
convex. For further discussion, see Appendix A of \cite{RSY04}.
\end{enumerate}
\end{remark}

The following lemma establishes some useful properties of a family
satisfying (KP1-3).

\begin{proposition}
\label{properties-of-KP}Let $\Lambda $ be a finitely aligned $k$-graph, $R$
be a commutative ring with $1$, and $\left\{ S_{\lambda },S_{\mu ^{\ast
}}:\lambda ,u\in \Lambda \right\} $ be a family satisfying (KP1-3) in an $R$%
-algebra $A$. Then

\begin{enumerate}
\item[(a)] $S_{\lambda }S_{\lambda ^{\ast }}S_{\mu }S_{\mu ^{\ast
}}=\sum_{\lambda \rho \in \operatorname{MCE}\left( \lambda ,\mu \right) }S_{\lambda
\rho }S_{\left( \lambda \rho \right) ^{\ast }}$ for $\lambda ,\mu \in
\Lambda $; and $\left\{ S_{\lambda }S_{\lambda ^{\ast }}:\lambda \in \Lambda
\right\} $ is a commuting family.

\item[(b)] The subalgebra generated by $\left\{ S_{\lambda },S_{\mu ^{\ast
}}:\lambda ,u\in \Lambda \right\} $ is
\begin{equation*}
\operatorname{span}_{R}\{S_{\lambda }S_{\mu ^{\ast }}:\lambda ,u\in \Lambda ,s\left(
\lambda \right) =s\left( \mu \right) \}\text{.}
\end{equation*}

\item[(c)] For $n\in \mathbb{N}^{k}$ and $\lambda ,\mu \in \Lambda ^{\leq n}$%
, we have $S_{\lambda ^{\ast }}S_{\mu }=\delta _{\lambda ,\mu }S_{s\left(
\lambda \right) }$.

\item[(d)] Suppose that $rS_{v}\neq 0$ for all $r\in \left.
R\right\backslash \left\{ 0\right\} $, $v\in \Lambda ^{0}$ and that $\lambda
,\mu \in \Lambda $ with $s\left( \lambda \right) =s\left( \mu \right) $. If $%
r\in \left. R\right\backslash \left\{ 0\right\} $ and $G\subseteq s\left(
\lambda \right) \Lambda $ is finite non-exhaustive, then%
\begin{equation*}
rS_{\lambda }\neq 0\text{ and }rS_{\lambda }\big(\prod_{\nu \in G}\left(
S_{s\left( \lambda \right) }-S_{\nu }S_{\nu ^{\ast }}\right) \big)S_{\mu
^{\ast }}\neq 0\text{.}
\end{equation*}
\end{enumerate}
\end{proposition}

\begin{proof}
To show (a), we take $\lambda ,\mu \in \Lambda $ and then%
\begin{align*}
S_{\lambda }S_{\lambda ^{\ast }}S_{\mu }S_{\mu ^{\ast }}& =S_{\lambda }\big(%
\sum_{(\rho ,\tau )\in \Lambda ^{\min }\left( \lambda ,\mu \right) }S_{\rho
}S_{\tau ^{\ast }}\big)S_{\mu ^{\ast }}=\sum_{(\rho ,\tau )\in \Lambda
^{\min }\left( \lambda ,\mu \right) }S_{\lambda \rho }S_{\left( \mu \tau
\right) ^{\ast }} \\
& =\sum_{(\rho ,\tau )\in \Lambda ^{\min }\left( \lambda ,\mu \right)
}S_{\lambda \rho }S_{\left( \lambda \rho \right) ^{\ast }}=\sum_{\lambda
\rho \in \operatorname{MCE}\left( \lambda ,\mu \right) }S_{\lambda \rho }S_{\left(
\lambda \rho \right) ^{\ast }}\text{.}
\end{align*}%
Furthermore,
\begin{equation*}
S_{\lambda }S_{\lambda ^{\ast }}S_{\mu }S_{\mu ^{\ast }}=\sum_{\lambda \rho
\in \operatorname{MCE}\left( \lambda ,\mu \right) }S_{\lambda \rho }S_{\left(
\lambda \rho \right) ^{\ast }}=\sum_{\mu \tau \in \operatorname{MCE}\left( \lambda
,\mu \right) }S_{\mu \tau }S_{\left( \mu \tau \right) ^{\ast }}=S_{\mu
}S_{\mu ^{\ast }}S_{\lambda }S_{\lambda ^{\ast }}\text{,}
\end{equation*}
as required.

Next we show (b). For $\lambda ,u\in \Lambda $, we have $S_{\lambda }S_{\mu
^{\ast }}=S_{\lambda }S_{s\left( \lambda \right) }S_{s\left( \mu \right)
}S_{\mu ^{\ast }}$ by (KP2). Then by (KP1), $S_{\lambda }S_{\mu ^{\ast
}}\neq 0$ implies $s\left( \lambda \right) =s\left( \mu \right) $.
Therefore, the result follows from part (a), (KP2) and (KP3).

To show (c), we take $\lambda ,\mu \in \Lambda ^{\leq n}$. Suppose that $%
S_{\lambda ^{\ast }}S_{\mu }\neq 0$. By (KP3), there exists $(\rho ,\tau
)\in \Lambda ^{\min }\left( \lambda ,\mu \right) $ such that $\lambda \rho
=\mu \tau $ and $d\left( \lambda \rho \right) \leq n$. Since $\lambda ,\mu
\in \Lambda ^{\leq n}$, then $\rho =s\left( \lambda \right) =\tau $ and
hence $\lambda =\mu $.

Finally, we show (d). Take $r\in \left. R\right\backslash \left\{ 0\right\} $
and $\lambda \in \Lambda $. Suppose for contradiction that $rS_{\lambda }=0$%
. Then%
\begin{equation*}
0=S_{\lambda ^{\ast }}\left( rS_{\lambda }\right) =rS_{\lambda ^{\ast
}}S_{\lambda }=rS_{s\left( \lambda \right) }\text{,}
\end{equation*}%
which contradicts with $rS_{v}\neq 0$ for all $r\in \left. R\right\backslash
\left\{ 0\right\} $ and $v\in \Lambda ^{0}$. Hence, $rS_{\lambda }\neq 0$.

Now take $r\in \left. R\right\backslash \left\{ 0\right\} $, $\lambda ,\mu
\in \Lambda $ with $s\left( \lambda \right) =s\left( \mu \right) $ and
finite non-exhaustive $G\subseteq s\left( \lambda \right) \Lambda $. Suppose
for contradiction that%
\begin{equation*}
rS_{\lambda }\big(\prod_{\nu \in G}\left( S_{s\left( \lambda \right)
}-S_{\nu }S_{\nu ^{\ast }}\right) \big)S_{\mu ^{\ast }}=0\text{.}
\end{equation*}%
Since $G$ is non-exhaustive, then there exists $\gamma \in s\left( \lambda
\right) \Lambda $ such that $\operatorname{Ext}\left( \gamma ;G\right) =\emptyset $.
Hence $\Lambda ^{\min }\left( \nu ,\gamma \right) =\emptyset $ for every $%
\nu \in G$, and then by (KP3), $S_{\nu ^{\ast }}S_{\gamma }=0$ for $\nu \in
G $. Therefore,%
\begin{align*}
0& =\big(rS_{\lambda }\big(\prod_{\nu \in G}\left( S_{s\left( \lambda
\right) }-S_{\nu }S_{\nu ^{\ast }}\right) \big)S_{\mu ^{\ast }}\big)S_{\mu
\gamma } \\
& =rS_{\lambda }\big(\prod_{\nu \in G}\left( S_{s\left( \lambda \right)
}-S_{\nu }S_{\nu ^{\ast }}\right) \big)S_{\gamma } \\
& =rS_{\lambda }S_{\gamma }=rS_{\lambda \gamma }\text{,}
\end{align*}%
which contradicts with $rS_{\lambda \gamma }\neq 0$. Hence, $rS_{\lambda }%
\big(\prod_{\nu \in G}\left( S_{s\left( \lambda \right) }-S_{\nu }S_{\nu
^{\ast }}\right) \big)S_{\mu ^{\ast }}\neq 0$.
\end{proof}

\begin{remark}
\label{properties-of-KP-additional}For $n\in \mathbb{N}^{k}$, we have $%
\Lambda ^{n}\subseteq \Lambda ^{\leq n}$. Hence, Proposition \ref%
{properties-of-KP}.(c) also implies that for $n\in \mathbb{N}^{k}$ and $%
\lambda ,\mu \in \Lambda ^{n}$, we have $S_{\lambda ^{\ast }}S_{\mu }=\delta
_{\lambda ,\mu }S_{s\left( \lambda \right) }$.
\end{remark}

\begin{remark}
\label{properties-of-KP-additional-2}Suppose that $rS_{v}\neq 0$ for all $%
r\in \left. R\right\backslash \left\{ 0\right\} $, $v\in \Lambda ^{0}$ and
that $\lambda ,\mu \in \Lambda $ with $s\left( \lambda \right) =s\left( \mu
\right) $. Then the contrapositive of Proposition \ref{properties-of-KP}.(d)
says: if $r\in R$ and $G\subseteq s\left( \lambda \right) \Lambda $ is
finite such that $rS_{\lambda }\big(\prod_{\nu \in G}\left( S_{s\left(
\lambda \right) }-S_{\nu }S_{\nu ^{\ast }}\right) \big)S_{\mu ^{\ast }}=0$,
then we have either $r=0$ or $G$ is exhaustive.
\end{remark}

Now we give an example of a Kumjian-Pask $\Lambda $-family in a particular
algebra of endomorphism.

\begin{proposition}
\label{the-boundary-path-representation}Let $\Lambda $ be a finitely aligned
$k$-graph and $R$ be a commutative ring with $1$. Let $\mathbb{F}_{R}\left(
\partial \Lambda \right) $ be the free module with basis the boundary path
space. Then for every $v\in \Lambda ^{0}$ and $\lambda ,\mu \in \left.
\Lambda \right\backslash \Lambda ^{0}$, there exist endomorphisms $%
S_{v},S_{\lambda },S_{\mu ^{\ast }}:\mathbb{F}_{R}\left( \partial \Lambda
\right) \rightarrow \mathbb{F}_{R}\left( \partial \Lambda \right) $ such
that for $x\in \partial \Lambda $,
\begin{align*}
S_{v}\left( x\right) & =%
\begin{cases}
x & \text{if }r\left( x\right) =v\text{;} \\
0 & \text{otherwise,}%
\end{cases}
\\
S_{\lambda }\left( x\right) & =%
\begin{cases}
\lambda x & \text{if }s\left( \lambda \right) =r\left( x\right) \text{;} \\
0 & \text{otherwise,}%
\end{cases}
\\
S_{\mu ^{\ast }}\left( x\right) & =%
\begin{cases}
\sigma ^{d\left( \mu \right) }x & \text{if }x\left( 0,d\left( \mu \right)
\right) =\mu \text{;} \\
0 & \text{otherwise.}%
\end{cases}%
\end{align*}%
Furthermore, $\left\{ S_{\lambda },S_{\mu ^{\ast }}:\lambda ,u\in \Lambda
\right\} $ is a Kumjian-Pask $\Lambda $-family in the $R$-algebra $\operatorname{End}%
\left( \mathbb{F}_{R}\left( \partial \Lambda \right) \right) $ with $%
rS_{v}\neq 0$ for all $r\in \left. R\right\backslash \left\{ 0\right\} $ and
$v\in \Lambda ^{0}$.
\end{proposition}

\begin{proof}
Take $v\in \Lambda ^{0}$ and $\lambda ,\mu \in \left. \Lambda
\right\backslash \Lambda ^{0}$. First note that for $x\in \partial \Lambda $
and $m\leq d\left( x\right) $, we have $\sigma ^{m}x\in \partial \Lambda $.
Now define functions $f_{v}$, $f_{\lambda }$, and $f_{\mu ^{\ast }}:\partial
\Lambda \rightarrow \mathbb{F}_{R}\left( \partial \Lambda \right) $ by%
\begin{align*}
f_{v}\left( x\right) & =%
\begin{cases}
x & \text{if }r\left( x\right) =v\text{;} \\
0 & \text{otherwise,}%
\end{cases}
\\
f_{\lambda }\left( x\right) & =%
\begin{cases}
\lambda x & \text{if }s\left( \lambda \right) =r\left( x\right) \text{;} \\
0 & \text{otherwise,}%
\end{cases}
\\
f_{\mu ^{\ast }}\left( x\right) & =%
\begin{cases}
\sigma ^{d\left( \mu \right) }x & \text{if }x\left( 0,d\left( \mu \right)
\right) =\mu \text{;} \\
0 & \text{otherwise.}%
\end{cases}%
\end{align*}%
The universal property of free modules gives nonzero endomorphisms
\begin{equation*}
S_{v},S_{\lambda },S_{\mu ^{\ast }}:\mathbb{F}_{R}\left( \partial \Lambda
\right) \rightarrow \mathbb{F}_{R}\left( \partial \Lambda \right)
\end{equation*}
extending $f_{v}$, $f_{\lambda }$, and $f_{\mu ^{\ast }}$, as needed.

Now we claim that $\left\{ S_{\lambda },S_{\mu ^{\ast }}:\lambda ,u\in
\Lambda \right\} $ is a Kumjian-Pask $\Lambda $-family. To see (KP1), take $%
v\in \Lambda ^{0}$ and $x\in \partial \Lambda $. Then we have $%
S_{v}^{2}\left( x\right) =x=S_{v}\left( x\right) $ if $r\left( x\right) =v$,
and $S_{v}^{2}\left( x\right) =0=S_{v}\left( x\right) $ otherwise. Hence $%
S_{v}^{2}=S_{v}$. Now take $v,w\in \Lambda ^{0}$ with $v\neq w$ and $x\in
\partial \Lambda $. Since $x\in w\partial \Lambda $ implies $x\notin
v\partial \Lambda $, we have $S_{v}S_{w}\left( x\right) =0$ for $x\in
\partial \Lambda $ and $S_{v}S_{w}=0$.

Next we show (KP2). Take $\lambda ,\mu \in \Lambda $ with $s\left( \lambda
\right) =r\left( \mu \right) $. Then for $x\in s\left( \mu \right) \partial
\Lambda $, we have $\mu x\in s\left( \lambda \right) \partial \Lambda $.
Then $S_{\lambda }S_{\mu }\left( x\right) =\lambda \mu x=S_{\lambda \mu
}\left( x\right) $ if $x\in s\left( \mu \right) \partial \Lambda $, and $%
S_{\lambda }S_{\mu }\left( x\right) =0=S_{\lambda \mu }\left( x\right) $
otherwise. Hence $S_{\lambda }S_{\mu }=S_{\lambda \mu }$. Meanwhile, for $%
x\in r\left( \lambda \right) \partial \Lambda $ with $x\left( 0,d\left(
\lambda \mu \right) \right) =\lambda \mu $, we have $d\left( \lambda \mu
\right) \leq d\left( x\right) $ and $\sigma ^{d\left( \lambda \mu \right)
}x\in s\left( \mu \right) \partial \Lambda $. Furthermore, $x\left(
0,d\left( \lambda \mu \right) \right) =\lambda \mu $, implies $x\left(
0,d\left( \lambda \right) \right) =\lambda $ and then we have $d\left(
\lambda \right) \leq d\left( x\right) $ and $\sigma ^{d\left( \lambda
\right) }x\in s\left( \lambda \right) \partial \Lambda $. Hence,
\begin{equation*}
S_{\mu ^{\ast }}S_{\lambda ^{\ast }}\left( x\right) =S_{\mu ^{\ast }}\sigma
^{d\left( \lambda \right) }x=\sigma ^{d\left( \lambda \right) +d\left( \mu
\right) }x=\sigma ^{d\left( \lambda \mu \right) }x=S_{\left( \lambda \mu
\right) ^{\ast }}\left( x\right)
\end{equation*}%
if $x\left( 0,d\left( \lambda \mu \right) \right) =\lambda \mu $, and $%
S_{\mu ^{\ast }}S_{\lambda ^{\ast }}\left( x\right) =0=S_{\left( \lambda \mu
\right) ^{\ast }}\left( x\right) $ otherwise. Therefore, $S_{\mu ^{\ast
}}S_{\lambda ^{\ast }}=S_{\left( \lambda \mu \right) ^{\ast }}$.

Now we show (KP3). Take $\lambda ,\mu \in \Lambda $. If $r\left( \lambda
\right) \neq r\left( \mu \right) $, then $S_{\lambda ^{\ast }}S_{\mu }=0$
and $\Lambda ^{\min }\left( \lambda ,\mu \right) =\emptyset $, as required.
Suppose $r\left( \lambda \right) =r\left( \mu \right) $. We have%
\begin{equation*}
S_{\lambda ^{\ast }}S_{\mu }\left( x\right) =%
\begin{cases}
\left( \mu x\right) \left( d\left( \lambda \right) ,d\left( \mu x\right)
\right) & \text{if }x\in s\left( \mu \right) \partial \Lambda \text{ and }%
\left( \mu x\right) \left( 0,d\left( \lambda \right) \right) =\lambda \text{;%
} \\
0 & \text{otherwise.}%
\end{cases}%
\end{equation*}%
Take $x\in s\left( \mu \right) \partial \Lambda $. Note that $s\left( \mu
\right) =r\left( \tau \right) $ for $(\rho ,\tau )\in \Lambda ^{\min }\left(
\lambda ,\mu \right) $. First suppose $\left( \mu x\right) \left( 0,d\left(
\lambda \right) \right) \neq \lambda $. Then for $(\rho ,\tau )\in \Lambda
^{\min }\left( \lambda ,\mu \right) $,%
\begin{equation*}
\left( \mu x\right) \left( 0,d\left( \lambda \rho \right) \right) \neq
\lambda \rho \text{ and }\left( \mu x\right) \left( 0,d\left( \mu \tau
\right) \right) \neq \mu \tau \text{.}
\end{equation*}%
Hence $x\left( 0,d\left( \tau \right) \right) \neq \tau $ and $S_{\rho
}S_{\tau ^{\ast }}\left( x\right) =S_{\rho }\left( 0\right) =0$. Therefore
\begin{equation*}
\sum_{(\rho ,\tau )\in \Lambda ^{\min }\left( \lambda ,\mu \right) }S_{\rho
}S_{\tau ^{\ast }}\left( x\right) =0.
\end{equation*}%
Next suppose $\left( \mu x\right) \left( 0,d\left( \lambda \right) \right)
=\lambda $. Since $\left( \mu x\right) \left( 0,d\left( \lambda \right)
\right) =\lambda $ and $\left( \mu x\right) \left( 0,d\left( \mu \right)
\right) =\mu $, there is $\tau \in s\left( \mu \right) \Lambda $ such that $%
(\rho ,\tau )\in \Lambda ^{\min }\left( \lambda ,\mu \right) $ and $\left(
\mu x\right) \left( 0,d\left( \mu \tau \right) \right) =\mu \tau $.
Therefore $x\left( 0,d\left( \tau \right) \right) =\tau $. Note that this $%
\tau $ is unique by the factorisation property. Hence for $(\rho ^{\prime
},\tau ^{\prime })\in \Lambda ^{\min }\left( \lambda ,\mu \right) $ such
that $(\rho ^{\prime },\tau ^{\prime })\neq (\rho ,\tau )$, we have $S_{\rho
^{\prime }}S_{\tau ^{\prime \ast }}\left( x\right) =0$. Also $x\left(
0,d\left( \tau \right) \right) =\tau $, thus%
\begin{align*}
S_{\rho }S_{\tau ^{\ast }}\left( x\right) & =S_{\rho }\left( x\left( d\left(
\tau \right) ,d\left( x\right) \right) \right) =\rho \left[ x\left( d\left(
\tau \right) ,d\left( x\right) \right) \right] \\
& =\rho \left[ \left( \mu x\right) \left( d\left( \mu \tau \right) ,d\left(
\mu x\right) \right) \right] \\
& =\rho \left[ \left( \mu x\right) \left( d\left( \lambda \rho \right)
,d\left( \mu x\right) \right) \right] \text{ (since }\mu \tau =\lambda \rho
\text{)} \\
& =\left( \mu x\right) \left( d\left( \lambda \right) ,d\left( \mu x\right)
\right)
\end{align*}%
and%
\begin{equation*}
\sum_{(\rho ^{\prime },\tau ^{\prime })\in \Lambda ^{\min }\left( \lambda
,\mu \right) }S_{\rho }S_{\tau ^{\ast }}\left( x\right) =S_{\rho }S_{\tau
^{\ast }}\left( x\right) =\left( \mu x\right) \left( d\left( \lambda \right)
,d\left( \mu x\right) \right) =S_{\lambda ^{\ast }}S_{\mu }\left( x\right)
\text{,}
\end{equation*}%
as required.

Finally, we show (KP4). Take $E\in \operatorname{FE}\left( \Lambda \right) $. Take $%
x\in r\left( E\right) \partial \Lambda $. Since $E\in x\left( 0\right) \operatorname{%
FE}\left( \Lambda \right) $ and $x$ is a boundary path, then there exists $%
\lambda \in E$ such that $x\left( 0,d\left( \lambda \right) \right) =\lambda
$. This implies%
\begin{align*}
\left( S_{r\left( E\right) }-S_{\lambda }S_{\lambda ^{\ast }}\right) \left(
x\right) & =S_{r\left( E\right) }\left( x\right) -S_{\lambda }S_{\lambda
^{\ast }}\left( x\right) \\
& =x-S_{\lambda }\left( x\left( d\left( \lambda \right) ,d\left( x\right)
\right) \right) \\
& =x-x=0\text{.}
\end{align*}%
Hence%
\begin{equation*}
\big(\prod_{\lambda \in E}\left( S_{r\left( E\right) }-S_{\lambda
}S_{\lambda ^{\ast }}\right) \big)\left( x\right) =0
\end{equation*}%
for $x\in r\left( E\right) \partial \Lambda $, and $\prod_{\lambda \in
E}\left( S_{r\left( E\right) }-S_{\lambda }S_{\lambda ^{\ast }}\right) =0$.

Thus $\left\{ S_{\lambda },S_{\mu ^{\ast }}:\lambda ,u\in \Lambda \right\} $
is a Kumjian-Pask $\Lambda $-family, as claimed. Now note that for $v\in
\Lambda ^{0}$, $v\partial \Lambda $ is nonempty. This implies that for all $%
r\in \left. R\right\backslash \left\{ 0\right\} $ and $v\in \Lambda ^{0}$, $%
rS_{v}\neq 0$.
\end{proof}

Using an alternate construction of a Kumjian-Pask $\Lambda $-family, we next
show that there is an $R$-algebra which is universal for Kumjian-Pask $%
\Lambda $-families.

\begin{theorem}
\label{universal-KP-family}Let $\Lambda $ be a finitely aligned $k$-graph
and $R$ be a commutative ring with $1$.

\begin{enumerate}
\item[(a)] There is a universal $R$-algebra ${\normalsize \operatorname{KP}}%
_{R}\left( \Lambda \right) $ generated by a Kumjian-Pask $\Lambda $-family $%
\{s_{\lambda },s_{\mu ^{\ast }}:\lambda ,u\in \Lambda \}$ such that whenever
$\left\{ S_{\lambda },S_{\mu ^{\ast }}:\lambda ,u\in \Lambda \right\} $ is a
Kumjian-Pask $\Lambda $-family in an $R$-algebra $A$, then there exists a
unique $R$-algebra homomorphism $\pi _{S}:{\normalsize \operatorname{KP}}_{R}\left(
\Lambda \right) \rightarrow A$ such that $\pi _{S}\left( s_{\lambda }\right)
=S_{\lambda }$ and $\pi _{S}\left( s_{\mu ^{\ast }}\right) =S_{\mu ^{\ast }}$
for $\lambda ,\mu \in \Lambda $.

\item[(b)] We have $rs_{v}\neq 0$ for all $r\in \left. R\right\backslash
\left\{ 0\right\} $ and $v\in \Lambda ^{0}$.

\item[(c)] The subsets%
\begin{equation*}
{\normalsize \operatorname{KP}}_{R}\left( \Lambda \right) _{n}:=\operatorname{span}%
_{R}\left\{ s_{\lambda }s_{\mu ^{\ast }}:\lambda ,\mu \in \Lambda ,d\left(
\lambda \right) -d\left( \mu \right) =n\right\}
\end{equation*}%
forms a $\mathbb{Z}^{k}$-grading of ${\normalsize \operatorname{KP}}_{R}\left(
\Lambda \right) $.
\end{enumerate}
\end{theorem}

\begin{proof}
We use an argument similar to \cite[Theorem 3.4]{ACaHR13} and \cite[Theorem
3.7]{CFaH14}. To show (a), first we define $X:=\Lambda \cup G\left( \Lambda
^{\neq 0}\right) $ and $\mathbb{F}_{R}\left( w\left( X\right) \right) $ be
the free algebra on the set $w\left( X\right) $ of words on $X$. Let $I$ be
the ideal of $\mathbb{F}_{R}\left( w\left( X\right) \right) $ generated by
elements of the following sets:

\begin{enumerate}
\item[(i)] $\left\{ vw-\delta _{v,w}v:v,w\in \Lambda ^{0}\right\} $,

\item[(ii)] $\{\lambda -\mu \nu ,\lambda ^{\ast }-\nu ^{\ast }\mu ^{\ast
}:\lambda ,\mu ,\nu \in \Lambda $ and $\lambda =\mu \nu \}$,

\item[(iii)] $\{\lambda ^{\ast }\mu -\sum_{(\rho ,\tau )\in \Lambda ^{\min
}\left( \lambda ,\mu \right) }\rho \tau ^{\ast }:\lambda ,\mu \in \Lambda \}$%
, and

\item[(iv)] $\{\prod_{\lambda \in E}\left( r\left( E\right) -\lambda \lambda
^{\ast }\right) :E\in \operatorname{FE}\left( \Lambda \right) \}$.
\end{enumerate}

Now define ${\normalsize \operatorname{KP}}_{R}\left( \Lambda \right) :=\mathbb{F}%
_{R}\left( w\left( X\right) \right) /I$ and $q:\mathbb{F}_{R}\left( w\left(
X\right) \right) \rightarrow \mathbb{F}_{R}\left( w\left( X\right) \right)
/I $ be the quotient map. Define $s_{\lambda }:=q\left( \lambda \right) $
for $\lambda \in \Lambda $, and $s_{\mu ^{\ast }}:=q\left( \mu ^{\ast
}\right) $ for $\mu ^{\ast }\in G\left( \Lambda ^{\neq 0}\right) $. Then $%
\{s_{\lambda },s_{\mu ^{\ast }}:\lambda \in \Lambda ,\mu ^{\ast }\in G\left(
\Lambda ^{\neq 0}\right) \}$ is a Kumjian-Pask $\Lambda $-family in $%
{\normalsize \operatorname{KP}}_{R}\left( \Lambda \right) $.

Now let $\left\{ S_{\lambda },S_{\mu ^{\ast }}:\lambda ,u\in \Lambda
\right\} $ be a Kumjian-Pask $\Lambda $-family in an $R$-algebra $A$. Define
$f:X\rightarrow A$ by $f\left( \lambda \right) :=S_{\lambda }$ for $\lambda
\in \Lambda $, and $f\left( \mu ^{\ast }\right) :=S_{\mu ^{\ast }}$ for $\mu
^{\ast }\in G\left( \Lambda ^{\neq 0}\right) $. The universal property of $%
\mathbb{F}_{R}\left( w\left( X\right) \right) $ gives an unique $R$-algebra
homomorphism $\phi :\mathbb{F}_{R}\left( w\left( X\right) \right)
\rightarrow A$ such that $\phi |_{X}=f$. Since $\left\{ S_{\lambda },S_{\mu
^{\ast }}:\lambda ,u\in \Lambda \right\} $ is a Kumjian-Pask $\Lambda $%
-family, then $I\subseteq \ker \left( \phi \right) $. Thus there exists an $%
R $-algebra homomorphism $\pi _{S}:{\normalsize \operatorname{KP}}_{R}\left( \Lambda
\right) \rightarrow A$ such that $\pi _{S}\circ q=\phi $. The homomorphism $%
\pi _{S}$ is unique since the elements in $X$ generate $\mathbb{F}_{R}\left(
w\left( X\right) \right) $ as an algebra. Furthermore, we have $\pi
_{S}\left( s_{\lambda }\right) =S_{\lambda }$ for $\lambda \in \Lambda $ and
$\pi _{S}\left( s_{\mu ^{\ast }}\right) =S_{\mu ^{\ast }}$ for $\mu ^{\ast
}\in G\left( \Lambda ^{\neq 0}\right) $, as required.

To show (b), let $\left\{ S_{\lambda },S_{\mu ^{\ast }}:\lambda ,u\in
\Lambda \right\} $ be the Kumjian-Pask $\Lambda $-family as in Proposition~%
\ref{the-boundary-path-representation}. Then $rS_{v}\neq 0$ for $v\in
\Lambda ^{0}$. Since $\pi _{S}\left( rs_{v}\right) =rS_{v}\neq 0$ for all $%
r\in \left. R\right\backslash \left\{ 0\right\} $ and $v\in \Lambda ^{0}$,
we have $rs_{v}\neq 0$ for all $r\in \left. R\right\backslash \left\{
0\right\} $ and $v\in \Lambda ^{0}$.

Next we show (c). We first extend the degree map to $w\left( X\right) $ by $%
d\left( w\right) :=\sum_{i=1}^{\left\vert w\right\vert }d\left( \left(
w_{i}\right) \right) $ for $w\in w\left( X\right) $. By \cite[Proposition 2.7%
]{ACaHR13}, $\mathbb{F}_{R}\left( w\left( X\right) \right) $ is $\mathbb{Z}%
^{k}$-graded by the subgroups%
\begin{equation*}
\mathbb{F}_{R}\left( w\left( X\right) \right) _{n}:=\left\{ \sum_{w\in
w\left( X\right) }r_{w}w:r_{w}\neq 0\text{ implies }d\left( w\right)
=n\right\} \text{.}
\end{equation*}

Now we claim that the ideal $I$ defined in (a) is a graded ideal. It
suffices to show that $I$ is generated by elements in $\mathbb{F}_{R}\left(
w\left( X\right) \right) _{n}$ for some $n\in \mathbb{Z}^{k}$. Since $%
d\left( v\right) =0$ for $v\in \Lambda ^{0}$, then the generators in (i)
belong to $\mathbb{F}_{R}\left( w\left( X\right) \right) _{0}$. If $\lambda
=\mu \nu $ in $\Lambda $, then $\lambda -\mu \nu $ belongs to $\mathbb{F}%
_{R}\left( w\left( X\right) \right) _{d\left( \lambda \right) }$ and $%
\lambda ^{\ast }-\nu ^{\ast }\mu ^{\ast }$ belongs to $\mathbb{F}_{R}\left(
w\left( X\right) \right) _{-d\left( \lambda \right) }$. For $\lambda ,\mu
\in \Lambda $ and $(\rho ,\tau )\in \Lambda ^{\min }\left( \lambda ,\mu
\right) $, we have%
\begin{equation*}
d\left( \rho \right) -d\left( \tau \right) =\left( d\left( \lambda \right)
\vee d\left( \mu \right) -d\left( \lambda \right) \right) -\left( d\left(
\lambda \right) \vee d\left( \mu \right) -d\left( \mu \right) \right)
=-d\left( \lambda \right) +d\left( \mu \right)
\end{equation*}%
and then the generators in (iii) belong to $\mathbb{F}_{R}\left( w\left(
X\right) \right) _{-d\left( \lambda \right) +d\left( \mu \right) }$.
Finally, a word $\lambda \lambda ^{\ast }$ has degree $0$ and then the
generators in (iv) belong to $\mathbb{F}_{R}\left( w\left( X\right) \right)
_{0}$. Thus $I$ is a graded ideal.

Since $I$ is graded, then ${\normalsize \operatorname{KP}}_{R}\left( \Lambda \right)
=\mathbb{F}_{R}\left( w\left( X\right) \right) /I$ is graded by the subgroups%
\begin{equation*}
\left( \mathbb{F}_{R}\left( w\left( X\right) \right) /I\right) _{n}:=\operatorname{%
span}_{R}\left\{ q\left( w\right) :w\in w\left( X\right) ,d\left( w\right)
=n\right\} \text{.}
\end{equation*}%
By Proposition \ref{properties-of-KP}.(b), we have ${\normalsize \operatorname{KP}}%
_{R}\left( \Lambda \right) =\operatorname{span}_{R}\left\{ s_{\lambda }s_{\mu ^{\ast
}}:\lambda ,u\in \Lambda ,s\left( \lambda \right) =s\left( \mu \right)
\right\} $. We have to show that%
\begin{equation*}
{\normalsize \operatorname{KP}}_{R}\left( \Lambda \right) _{n}:=\operatorname{span}%
_{R}\left\{ s_{\lambda }s_{\mu ^{\ast }}:\lambda ,u\in \Lambda ,d\left(
\lambda \right) -d\left( \mu \right) =n\right\} =\left( \mathbb{F}_{R}\left(
w\left( X\right) \right) /I\right) _{n}\text{.}
\end{equation*}%
Take $\lambda ,u\in \Lambda $ with $d\left( \lambda \right) -d\left( \mu
\right) =n$. Then $s_{\lambda }s_{\mu ^{\ast }}=q\left( \lambda \right)
q\left( \mu ^{\ast }\right) =q\left( \lambda \mu ^{\ast }\right) $ and $%
d\left( \lambda \mu ^{\ast }\right) =d\left( \lambda \right) -d\left( \mu
\right) =n$. Hence $s_{\lambda }s_{\mu ^{\ast }}\in \left( \mathbb{F}%
_{R}\left( w\left( X\right) \right) /I\right) _{n}$ and ${\normalsize \operatorname{%
KP}}_{R}\left( \Lambda \right) _{n}\subseteq \left( \mathbb{F}_{R}\left(
w\left( X\right) \right) /I\right) _{n}$.

To prove $\left( \mathbb{F}_{R}\left( w\left( X\right) \right) /I\right)
_{n}\subseteq {\normalsize \operatorname{KP}}_{R}\left( \Lambda \right) _{n} $, we
first establish the following claim:

\begin{claim}
\label{Ff(w(X))-in-KP} Let $X:=\Lambda \cup G\left( \Lambda ^{\neq 0}\right)
$ and $q:\mathbb{F}_{R}\left( w\left( X\right) \right) \rightarrow
{\normalsize \operatorname{KP}}_{R}\left( \Lambda \right) $ be the quotient map.
Then for $w\in w\left( X\right) $, we have $q\left( w\right) \in
{\normalsize \operatorname{KP}}_{R}\left( \Lambda \right) _{d\left( w\right) }$.
\end{claim}

\begin{proof}[Proof of Claim~\protect\ref{Ff(w(X))-in-KP}]
We are modifying the proof of \cite[Lemma 3.5]{ACaHR13} and \cite[Lemma 3.8]%
{CFaH14} using our version of (KP3). We prove the claim by induction on $%
\left\vert w\right\vert $. For $\left\vert w\right\vert =0$, we have $w=v$
for some $v\in \Lambda ^{0}$. Then $q\left( w\right) =s_{v}=s_{v}s_{v^{\ast
}}$ and $d\left( v\right) -d\left( v\right) =0$. So $q\left( w\right) \in
{\normalsize \operatorname{KP}}_{R}\left( \Lambda \right) _{d\left( w\right) }$.

For $\left\vert w\right\vert =1$, we have two possibilities. First suppose $%
w=\lambda $ for $\lambda \in \Lambda $. Then $q\left( w\right) =s_{\lambda
}=s_{\lambda }s_{s\left( \lambda \right) ^{\ast }}$ and $d\left( \lambda
\right) -d\left( s\left( \lambda \right) \right) =d\left( \lambda \right) $.
So $q\left( w\right) \in {\normalsize \operatorname{KP}}_{R}\left( \Lambda \right)
_{d\left( w\right) }$. Next suppose $w=\lambda ^{\ast }$ for $\lambda \in
\Lambda $. Then $q\left( w\right) =s_{\lambda ^{\ast }}=s_{s\left( \lambda
\right) }s_{\lambda ^{\ast }}$ and $d\left( s\left( \lambda \right) \right)
-d\left( \lambda \right) =-d\left( \lambda \right) =d\left( \lambda ^{\ast
}\right) $. So $q\left( w\right) \in {\normalsize \operatorname{KP}}_{R}\left(
\Lambda \right) _{d\left( w\right) }$.

For $\left\vert w\right\vert =2$, we have four possibilities: $w=\lambda \mu
^{\ast }$, $w=\lambda \mu $, $w=\mu ^{\ast }\lambda ^{\ast }$, or $w=\lambda
^{\ast }\mu $. For the first three cases, we have%
\begin{align*}
& q\left( \lambda \mu ^{\ast }\right) =s_{\lambda }s_{\mu ^{\ast }}\text{
and }d\left( \lambda \right) -d\left( \mu \right) =d\left( \lambda \mu
^{\ast }\right) \text{,} \\
& q\left( \lambda \mu \right) =s_{\lambda \mu }s_{s\left( \mu \right) ^{\ast
}}\text{ and }d\left( \lambda \mu \right) -d\left( s\left( \mu \right)
\right) =d\left( \lambda \mu \right) \text{,} \\
& q\left( \mu ^{\ast }\lambda ^{\ast }\right) =s_{s\left( \mu \right)
}s_{\left( \lambda \mu \right) ^{\ast }}\text{ and }d\left( s\left( \mu
\right) \right) -d\left( \left( \lambda \mu \right) ^{\ast }\right) =d\left(
\mu ^{\ast }\lambda ^{\ast }\right) \text{,}
\end{align*}%
as required. Suppose $w=\lambda ^{\ast }\mu $. By (KP3), we have%
\begin{equation*}
q\left( \lambda ^{\ast }\mu \right) =s_{\lambda ^{\ast }}s_{\mu
}=\sum_{(\rho ,\tau )\in \Lambda ^{\min }\left( \lambda ,\mu \right)
}s_{\rho }s_{\tau ^{\ast }}\text{.}
\end{equation*}%
For $(\rho ,\tau )\in \Lambda ^{\min }\left( \lambda ,\mu \right) $, we have
$\lambda \rho =\mu \tau $ and then $d\left( w\right) =d\left( \mu \right)
-d\left( \lambda \right) =d\left( \rho \right) -d\left( \rho \right) $. So $%
q\left( w\right) \in {\normalsize \operatorname{KP}}_{R}\left( \Lambda \right)
_{d\left( w\right) }$.

Now suppose that $n\geq 2$ and $q\left( y\right) \in {\normalsize \operatorname{KP}}%
_{R}\left( \Lambda \right) _{d\left( y\right) }$ for every word $y$ with $%
\left\vert y\right\vert \leq n$. Let $w$ be a word with $\left\vert
w\right\vert =n+1$ and $q\left( w\right) \neq 0$. If $w$ contains a subword $%
w_{i}w_{i+1}=\lambda \mu $, then $\lambda $ and $\mu $ are composable in $%
\Lambda $ since otherwise $q\left( \lambda \mu \right) =0$. Now let $%
w^{\prime }$ be the word obtained from $w$ by replacing $w_{i}w_{i+1}$ with
the single path $\lambda \mu $, and then%
\begin{equation*}
q\left( w\right) =s_{w_{1}}\cdots s_{w_{i-1}}s_{\lambda }s_{\mu
}s_{w_{i+2}}s_{w_{n+1}}=s_{w_{1}}\cdots s_{w_{i-1}}s_{\lambda \mu
}s_{w_{i+2}}s_{w_{n+1}}=q\left( w^{\prime }\right) \text{.}
\end{equation*}%
Since $\left\vert w^{\prime }\right\vert =n$ and $d\left( w^{\prime }\right)
=d\left( w\right) $, the inductive hypothesis implies $q\left( w\right) \in
{\normalsize \operatorname{KP}}_{R}\left( \Lambda \right) _{d\left( w\right) }$. A
similar argument shows $q\left( w\right) \in {\normalsize \operatorname{KP}}%
_{R}\left( \Lambda \right) _{d\left( w\right) }$ whenever $w$ contains a
subword $w_{i}w_{i+1}=\mu ^{\ast }\lambda ^{\ast }$.

So suppose $w$ contains no subword of the form $\lambda \mu $ or $\mu ^{\ast
}\lambda ^{\ast }$. Since $\left\vert w\right\vert \geq 3$, either $%
w_{1}w_{2}$ or $w_{2}w_{3}$ has the form $\lambda ^{\ast }\mu $. By (KP3),
we write $q\left( w\right) $ as a sum of terms $q\left( y^{i}\right) $ with $%
\left\vert y^{i}\right\vert =n+1$ and $d\left( y^{i}\right) =d\left(
w\right) $. Since $\left\vert w\right\vert \geq 3$, each nonzero summand $%
q\left( y^{i}\right) $ contains a factor of the form $s_{\gamma }s_{\rho }$
or one of the form $s_{\tau ^{\ast }}s_{\gamma ^{\ast }}$. Then the previous
argument shows that every $q\left( y^{i}\right) \in {\normalsize \operatorname{KP}}%
_{R}\left( \Lambda \right) _{d\left( w\right) }$ and $q\left( w\right) \in
{\normalsize \operatorname{KP}}_{R}\left( \Lambda \right) _{d\left( w\right) }$, as
required. \hfil\penalty100\hbox{}\nobreak\hfill%
\hbox{\qed\
Claim~\ref{Ff(w(X))-in-KP}} \renewcommand\qed{}
\end{proof}

Every element in $\left( \mathbb{F}_{R}\left( w\left( X\right) \right)
/I\right) _{n}$ is in the form $q\left( w\right) $ with $w\in w\left(
X\right) $ and $d\left( w\right) =n$, which, by Claim \ref{Ff(w(X))-in-KP},\
belongs to ${\normalsize \operatorname{KP}}_{R}\left( \Lambda \right) _{n}$. Then $%
\left( \mathbb{F}_{R}\left( w\left( X\right) \right) /I\right) _{n}\subseteq
{\normalsize \operatorname{KP}}_{R}\left( \Lambda \right) _{n} $, as required.
\end{proof}

\begin{definition}
Suppose that $\left\{ S_{\lambda },S_{\mu ^{\ast }}:\lambda ,u\in \Lambda
\right\} $ is the Kumjian-Pask $\Lambda $-family in the $R$-algebra $\operatorname{%
End}\left( \mathbb{F}_{R}\left( \partial \Lambda \right) \right) $ as in\
Proposition \ref{the-boundary-path-representation}. We call the $R$-algebra
homomorphism $\pi _{S}:{\normalsize \operatorname{KP}}_{R}\left( \Lambda \right)
\rightarrow \operatorname{End}\left( \mathbb{F}_{R}\left( \partial \Lambda \right)
\right) $ obtained from Theorem \ref{universal-KP-family}.(a), the \emph{%
boundary path representation of }${\normalsize \operatorname{KP}}_{R}\left( \Lambda
\right) $.
\end{definition}

\section{The graded-invariant uniqueness theorem}

\label{Section-the-graded-uniqueness-theorem}Throughout this section, $%
\Lambda $ is a finitely aligned $k$-graph and $R$ is a commutative ring with
identity $1$.

\begin{theorem}[The graded-uniqueness theorem]
\label{the-graded-uniqueness-theorem}Let $\Lambda $ be a finitely aligned $k$%
-graph, $R$ be a commutative ring with $1$, and $A$ be a $\mathbb{Z}^{k}$%
-graded $R$-algebra. Suppose that $\pi :{\normalsize \operatorname{KP}}_{R}\left(
\Lambda \right) \rightarrow A$ is a $\mathbb{Z}^{k}$-graded $R$-algebra
homomorphism such that $\pi \left( rs_{v}\right) \neq 0$ for all $r\in
\left. R\right\backslash \left\{ 0\right\} $ and $v\in \Lambda ^{0}$. Then $%
\pi $ is injective.
\end{theorem}

We start the proof of Theorem \ref{the-graded-uniqueness-theorem} by
adapting some $C^{\ast }$-algebra results used to prove the gauge-invariant
uniqueness theorem \cite[Theorem 4.2]{RSY04} to Kumjian-Pask algebras.
Although the argument is rather technical, the point is that most of the
argument in $C^{\ast}$-algebra setting also works in our situation.

First we recall from \cite[Definition 2.5]{RSY04} that a \emph{Cuntz-Krieger
}$\Lambda $\emph{-family} is a collection $\left\{ T_{\lambda }:\lambda \in
\Lambda \right\} $ of partial isometries (in other words, it satisfies $%
T_{\lambda }=T_{\lambda }T_{\lambda }^{\ast }T_{\lambda }$ for $\lambda \in
\Lambda $, see \cite[Appendix A]{CBMS}) in a $C^{\ast }$-algebra $B$
satisfying:

\begin{enumerate}
\item[(TCK1)] $\left\{ T_{v}:v\in \Lambda ^{0}\right\} $ is a collection of
mutually orthogonal projections;

\item[(TCK2)] $T_{\lambda }T_{\mu }=T_{\lambda \mu }$ whenever $s\left(
\lambda \right) =r\left( \mu \right) $;

\item[(TCK3)] $T_{\lambda }^{\ast }T_{\mu }=\sum_{(\rho ,\tau )\in \Lambda
^{\min }\left( \lambda ,\mu \right) }T_{\rho }T_{\tau }^{\ast }$ for all $%
\lambda ,\mu \in \Lambda $; and

\item[(CK)] $\prod_{\lambda \in E}\left( T_{r\left( E\right) }-T_{\lambda
}T_{\lambda }^{\ast }\right) =0$ for all $E\in \operatorname{FE}\left( \Lambda
\right) $.
\end{enumerate}

For a finitely aligned $k$-graph $\Lambda $, there exists a universal $%
C^{\ast }$-algebra $C^{\ast }\left( \Lambda \right) $ generated by the
universal Cuntz-Krieger $\Lambda $-family $\left\{ t_{\lambda }:\lambda \in
\Lambda \right\} $. Now suppose that $\left\{ S_{\lambda },S_{\mu ^{\ast
}}:\lambda ,u\in \Lambda \right\} $ is a Kumjian-Pask $\Lambda $-family in
an $R$-algebra $A$ and we define $T_{\lambda }:=S_{\lambda }$ for $\lambda
\in \Lambda $ and $T_{\mu }^{\ast }:=S_{\mu ^{\ast }}$ for $\mu \in G\left(
\Lambda ^{\neq 0}\right) $. Then $\left\{ T_{\lambda }:\lambda \in \Lambda
\right\} $ is a collection satisfying $T_{\lambda }=T_{\lambda }T_{\lambda
}^{\ast }T_{\lambda }$ for $\lambda \in \Lambda $, (TCK1-3) and (CK). (Note
that we do not say that $\left\{ T_{\lambda }:\lambda \in \Lambda \right\} $
is a Cuntz-Krieger $\Lambda $-family, since we need a $C^{\ast }$-algebra
containing $T_{\lambda },T_{\mu }^{\ast }$.) Similarly, a Cuntz-Krieger $%
\Lambda $-family in a $C^{\ast }$-algebra gives a Kumjian-Pask $\Lambda $%
-family. Thus one can translate proofs about Cuntz-Krieger $\Lambda $%
-families to proofs about Kumjian-Pask $\Lambda $-families.

The key ingredient to proof of Theorem \ref{the-graded-uniqueness-theorem}
is proving that the uniqueness theorem holds on the core ${\normalsize \operatorname{%
KP}}_{R}\left( \Lambda \right) _{0}:=\operatorname{span}_{R}\left\{ s_{\lambda
}s_{\mu ^{\ast }}:d\left( \lambda \right) =d\left( \mu \right) \right\} $
(Theorem \ref{injectivity-on-the core}). First we establish some preliminary
results and notation.

Following \cite[Lemma 3.2]{RSY04}, for every finite set $E\subseteq \Lambda $%
, there exists a finite set $F\subseteq \Lambda $ which contains $E$ and
satisfies%
\begin{align}
\lambda ,\mu ,\rho ,\tau \in & F\text{, }d\left( \lambda \right) =d\left(
\mu \right) \text{, }d\left( \rho \right) =d\left( \tau \right) \text{, }%
s\left( \lambda \right) =s\left( \mu \right) \text{, and }s\left( \rho
\right) =s\left( \tau \right) \text{ }  \label{pi-E} \\
& \text{imply }\left\{ \lambda \alpha ,\tau \beta :\left( \alpha ,\beta
\right) \in \Lambda ^{\min }\left( \mu ,\rho \right) \right\} \subseteq F%
\text{.}  \notag
\end{align}%
We then write
\begin{equation*}
\Pi E:=\bigcap \{F\subseteq \Lambda :E\subseteq F\text{ and }F\text{
satisfies \eqref{pi-E}}\}
\end{equation*}%
and $\Pi E\times _{d,s}\Pi E$ for the set $\left\{ \left( \lambda ,\mu
\right) \in \Pi E\times \Pi E:d\left( \lambda \right) =d\left( \mu \right)
,s\left( \lambda \right) =s\left( \mu \right) \right\} $. Note that $\Pi E$
is finite. Now recall from Notation 3.12 of \cite{RSY04} that for $\lambda
\in \Pi E$, we write%
\begin{equation*}
T\left( \lambda \right) :=\left\{ \nu \in s\left( \lambda \right) \Lambda
:d\left( \nu \right) \neq 0,\lambda \nu \in \Pi E\right\} \text{.}
\end{equation*}%
Since $\lambda T\left( \lambda \right) \subseteq \Pi E$ and $\Pi E$ is
finite, then $T\left( \lambda \right) $ is also finite.

Now suppose that $\left\{ S_{\lambda },S_{\mu ^{\ast}}:\lambda ,u\in \Lambda
\right\}$ is a Kumjian-Pask $\Lambda $-family in an $R$-algebra $A$. The
argument of Lemma 3.2 of \cite{RSY04} shows that the set%
\begin{equation*}
M_{\Pi E}^{S}:=\operatorname{span}_{R}\left\{ S_{\lambda }S_{\mu ^{\ast }}:\left(
\lambda ,\mu \right) \in \Pi E\times _{d,s}\Pi E\right\}
\end{equation*}%
is closed under multiplication. For $\left( \lambda ,\mu \right) \in \Pi
E\times _{d,s}\Pi E$, define%
\begin{equation*}
\Theta \left( S\right) _{\lambda ,\mu }^{\Pi E}:=S_{\lambda }\big(\prod_{\nu
\in T\left( \lambda \right) }\left( S_{s\left( \lambda \right) }-S_{\lambda
\nu }S_{\left( \lambda \nu \right) ^{\ast }}\right) \big)S_{\mu ^{\ast }%
\text{.}}
\end{equation*}

Applying the argument of Proposition 3.9 and Proposition 3.11 of \cite{RSY04}
gives the following.

\begin{lemma}
\label{span-of-M}Let $\left\{ S_{\lambda },S_{\mu ^{\ast }}:\lambda ,u\in
\Lambda \right\} $ be a Kumjian-Pask $\Lambda $-family in an $R$-algebra $A$
and $E \subseteq \Lambda$ be finite. For $\left( \lambda ,\mu \right)
,\left( \rho ,\tau \right) \in \Pi E\times _{d,s}\Pi E$, we have
\begin{equation*}
\Theta \left( S\right) _{\lambda ,\mu }^{\Pi E}\Theta \left( S\right) _{\rho
,\tau }^{\Pi E}=\delta _{\mu ,\rho }\Theta \left( S\right) _{\lambda ,\tau
}^{\Pi E}, \quad S_{\lambda }S_{\mu ^{\ast }}=\sum_{\lambda \nu \in \Pi
E}\Theta \left( S\right) _{\lambda \nu ,\mu \nu }^{\Pi E}
\end{equation*}%
and $M_{\Pi E}^{S}$ is spanned by the set $\{\Theta \left( S\right)
_{\lambda ,\mu }^{\Pi E}:\left( \lambda ,\mu \right) \in \Pi E\times
_{d,s}\Pi E\}$.
\end{lemma}

\begin{lemma}
\label{injectivity-on-M}Let $\Lambda $ be a finitely aligned $k$-graph, $R $
be a commutative ring with $1$ and $E \subseteq \Lambda$ be finite. Suppose
that $\pi :{\normalsize \operatorname{KP}}_{R}\left( \Lambda \right) \rightarrow A$
is an $R$-algebra homomorphism such that $\pi \left( rs_{v}\right) \neq 0$
for all $r\in \left. R\right\backslash \left\{ 0\right\} $ and $v\in \Lambda
^{0}$. Let $\left( \lambda ,\mu \right) \in \Pi E\times _{d,s}\Pi E$. Then
the following conditions are equivalent:

\begin{enumerate}
\item[(a)] $\pi \big(\Theta \left( s\right) _{\lambda ,\mu }^{\Pi E}\big)=0$.

\item[(b)] $\Theta \left( s\right) _{\lambda ,\mu }^{\Pi E}=0$.

\item[(c)] $T\left( \lambda \right) $ is exhaustive.
\end{enumerate}

Furthermore, for $r\in \left. R\right\backslash \left\{ 0\right\} $ we have%
\begin{equation*}
\pi \big(r\Theta \left( s\right) _{\lambda ,\mu }^{\Pi E}\big)=0\text{ if
and only if }r\Theta \left( s\right) _{\lambda ,\mu }^{\Pi E}=0
\end{equation*}%
and $\pi $ is injective on $M_{\Pi E}^{s}$.
\end{lemma}

\begin{proof}
By following the argument of Proposition 3.13 and Corollary 3.17 of \cite%
{RSY04}, we have the three equivalent conditions. Now take $\left( \lambda
,\mu \right) \in \Pi E\times _{d,s}\Pi E$ and $r\in \left. R\right\backslash
\left\{ 0\right\} $. If $r\Theta \left( s\right) _{\lambda ,\mu }^{\Pi E}=0$%
, we trivially have $\pi \big(r\Theta \left( s\right) _{\lambda ,\mu }^{\Pi
E}\big)=0$. So suppose $\pi \big(r\Theta \left( s\right) _{\lambda ,\mu
}^{\Pi E}\big)=0$. Since $\pi \left( rs_{v}\right) \neq 0$ for all $r\in
\left. R\right\backslash \left\{ 0\right\} $ and $v\in \Lambda ^{0}$, then
by Remark \ref{properties-of-KP-additional-2}, $\pi \big(r\Theta \left(
s\right) _{\lambda ,\mu }^{\Pi E}\big)=0$ implies that $T\left( \lambda
\right) $ is exhaustive (since $r\neq 0$) and by (c)$\Rightarrow $(b), $%
\Theta \left( s\right) _{\lambda ,\mu }^{\Pi E}=0$. So $r\Theta \left(
s\right) _{\lambda ,\mu }^{\Pi E}=0$, as required.

Next we show that $\pi $ is injective on $M_{\Pi E}^{s}$. Take $a\in M_{\Pi
E}^{s}$ such that $\pi \left( a\right) =0$. We have to show $a=0$. Since $%
a\in M_{\Pi E}^{s}$ and $M_{\Pi E}^{s}=\operatorname{span}_{R}\{\Theta \left(
s\right) _{\lambda ,\mu }^{\Pi E}:\left( \lambda ,\mu \right) \in \Pi
E\times _{d,s}\Pi E\}$ (Lemma \ref{span-of-M}), we write $a=\sum_{\left(
\lambda ,\mu \right) \in F}r_{\lambda ,\mu }\Theta \left( s\right) _{\lambda
,\mu }^{\Pi E}$ where $F\subseteq \Pi E\times _{d,s}\Pi E$ is finite and for
all $\left( \lambda ,\mu \right) \in F$, we have $r_{\lambda ,\mu }\in R$
and $\Theta \left( s\right) _{\lambda ,\mu }^{\Pi E}\neq 0$. If $T\left(
\lambda \right) $ is exhaustive for some $\left( \lambda ,\mu \right) \in F$%
, then by (c)$\Rightarrow $(b), $\Theta \left( s\right) _{\lambda ,\mu
}^{\Pi E}=0$, which contradicts $\Theta \left( s\right) _{\lambda ,\mu
}^{\Pi E}\neq 0$. So $T\left( \lambda \right) $ is non-exhaustive for all $%
\left( \lambda ,\mu \right) \in F$. Since $\pi \left( a\right) =0$, then for
$\left( \rho ,\tau \right) \in F$, we have
\begin{align*}
0& =\pi \big(\Theta \left( s\right) _{\rho ,\rho }^{\Pi E}\big)\pi \left(
a\right) \pi \big(\Theta \left( s\right) _{\tau ,\tau }^{\Pi E}\big) \\
& =\pi \big(\Theta \left( s\right) _{\rho ,\rho }^{\Pi E}\big)\pi \big(%
\sum_{\left( \lambda ,\mu \right) \in F}r_{\lambda ,\mu }\Theta \left(
s\right) _{\lambda ,\mu }^{\Pi E}\big)\pi \big(\Theta \left( s\right) _{\tau
,\tau }^{\Pi E}\big) \\
& =r_{\rho ,\tau }\pi \big(\Theta \left( s\right) _{\rho ,\tau }^{\Pi E}\big)
= r_{\rho ,\tau } \Theta
\left( \pi\left(s\right)\right) _{\rho ,\tau }^{\Pi E}%
\text{ (by Lemma \ref{span-of-M}).}
\end{align*}%
But now since $\pi \left( rs_{v}\right) \neq 0$ for all $r\in \left.
R\right\backslash \left\{ 0\right\} $ and $v\in \Lambda ^{0}$, then by
Remark \ref{properties-of-KP-additional-2}, $r_{\rho ,\tau } \Theta
\left( \pi\left(s\right)\right) _{\rho ,\tau }^{\Pi E}=0$ implies that $r_{\rho ,\tau
}=0$ (since $T\left( \rho \right) $ is non-exhaustive). Therefore, $a=0$ and
$\pi $ is injective on $M_{\Pi E}^{s}$.
\end{proof}

A direct consequence of Lemma \ref{injectivity-on-M} is:

\begin{theorem}
\label{injectivity-on-the core}Let $\Lambda $ be a finitely aligned $k$%
-graph and $R$ be a commutative ring with $1$. Suppose that $\pi :%
{\normalsize \operatorname{KP}}_{R}\left( \Lambda \right) \rightarrow A$ is an $R$%
-algebra homomorphism such that $\pi \left( rs_{v}\right) \neq 0$ for all $%
r\in \left. R\right\backslash \left\{ 0\right\} $ and $v\in \Lambda ^{0}$.
Then $\pi $ is injective on ${\normalsize \operatorname{KP}}_{R}\left( \Lambda
\right) _{0}$.
\end{theorem}

\begin{proof}
Take $a\in {\normalsize \operatorname{KP}}_{R}\left( \Lambda \right) _{0}$ such that
$\pi \left( a\right) =0$. We have to show $a=0$. Write $a=\sum_{\left(
\lambda ,\mu \right) \in F}r_{\lambda ,\mu }s_{\lambda }s_{\mu ^{\ast }}$
with $d\left( \lambda \right) =d\left( \mu \right) $ for $\left( \lambda
,\mu \right) \in F$. Define $E:=\left\{ \lambda ,\mu :\left( \lambda ,\mu
\right) \in F\right\} $ and then $a\in M_{\Pi E}^{s}$. Since $\pi $ is
injective on $M_{\Pi E}^{s}$ (Lemma \ref{injectivity-on-M}), $a=0$.
\end{proof}

Now we establish the last stepping stone result before proving Theorem \ref%
{the-graded-uniqueness-theorem}.

\begin{lemma}
\label{generator-of-ideal}Let $I$ be a graded ideal of ${\normalsize \operatorname{KP%
}}_{R}\left( \Lambda \right) $. Then $I$ is generated as an ideal by the set
$I_{0}:=I\cap {\normalsize \operatorname{KP}}_{R}\left( \Lambda \right) _{0}$.
\end{lemma}

\begin{proof}
We generalise the argument of \cite[Lemma 5.1]{T11}. Take $n\in \mathbb{Z}%
^{k}$ and write $n=n_{1}-n_{2}$ such that $n_{1},n_{2}\in \mathbb{N}^{k}$.
We show that $I_{n}:=I\cap {\normalsize \operatorname{KP}}_{R}\left( \Lambda \right)
_{n}$ is contained in ${\normalsize \operatorname{KP}}_{R}\left( \Lambda \right)
_{n_{1}}I_{0}{\normalsize \operatorname{KP}}_{R}\left( \Lambda \right) _{n_{2}}$.
Now take $a\in I_{n}$ and write $a=\sum_{\left( \lambda ,\mu \right) \in
F}r_{\lambda ,\mu }s_{\lambda }s_{\mu ^{\ast }}$. Note that $d\left( \lambda
\right) -d\left( \mu \right) =n$ for $\left( \lambda ,\mu \right) \in F$.
Since $n=n_{1}-n_{2}$ with $n_{1},n_{2}\in \mathbb{N}^{k}$, then for every $%
\left( \lambda ,\mu \right) \in F$, by the factorisation property, there
exist $\lambda _{1},\lambda _{2},\mu _{1},\mu _{2}$ such that%
\begin{equation*}
\lambda =\lambda _{1}\lambda _{2}\text{, }\mu =\mu _{1}\mu _{2}\text{, }%
d\left( \lambda _{1}\right) =n_{1}\text{, }d\left( \mu _{1}\right) =n_{2}%
\text{, and }d\left( \lambda _{2}\right) =d\left( \mu _{2}\right) \text{.}
\end{equation*}%
Hence $a=\sum_{\left( \lambda _{1}\lambda _{2},\mu _{1}\mu _{2}\right) \in
F}r_{\lambda _{1}\lambda _{2},\mu _{1}\mu _{2}}s_{\lambda _{1}}\left(
s_{\lambda _{2}}s_{\mu _{2}^{\ast }}\right) s_{\mu _{1}^{\ast }}$. Take $%
\left( \alpha _{1}\alpha _{2},\beta _{1}\beta _{2}\right) \in F$. \ Note
that for $\nu ,\gamma \in \Lambda $ with $d\left( \nu \right) =d\left(
\gamma \right) $, by Remark \ref{properties-of-KP-additional}, we have $%
s_{\nu ^{\ast }}s_{\gamma }=0$ if $\nu \neq \gamma $. Then%
\begin{align*}
s_{\alpha _{1}^{\ast }}as_{\beta _{1}}& =\sum_{\left( \lambda _{1}\lambda
_{2},\mu _{1}\mu _{2}\right) \in F}r_{\lambda _{1}\lambda _{2},\mu _{1}\mu
_{2}}\left( s_{\alpha _{1}^{\ast }}s_{\lambda _{1}}\right) \left( s_{\lambda
_{2}}s_{\mu _{2}^{\ast }}\right) \left( s_{\mu _{1}^{\ast }}s_{\beta
_{1}}\right) \\
& =\sum_{\left( \alpha _{1}\lambda _{2},\beta _{1}\mu _{2}\right) \in
F}r_{\alpha _{1}\lambda _{2},\beta _{1}\mu _{2}}s_{\lambda _{2}}s_{\mu
_{2}^{\ast }}
\end{align*}%
since $d\left( \alpha _{1}\right) =n_{1}=d\left( \lambda _{1}\right) $ and $%
d\left( \beta _{1}\right) =n_{2}=d\left( \mu _{1}\right) $ for $\left(
\lambda _{1}\lambda _{2},\mu _{1}\mu _{2}\right) \in F$. Since $a\in I$,
then $s_{\alpha _{1}^{\ast }}as_{\beta _{1}}\in I$. Furthermore, since $%
d\left( \lambda _{2}\right) =d\left( \mu _{2}\right) $ for $\left( \alpha
_{1}\lambda _{2},\beta _{1}\mu _{2}\right) \in F$, then $s_{\alpha
_{1}^{\ast }}as_{\beta _{1}}\in {\normalsize \operatorname{KP}}_{R}\left( \Lambda
\right) _{0}$. Hence, for $\left( \alpha _{1}\alpha _{2},\beta _{1}\beta
_{2}\right) \in F$,
\begin{equation*}
\sum_{\left( \alpha _{1}\lambda _{2},\beta _{1}\mu _{2}\right) \in
F}r_{\alpha _{1}\lambda _{2},\beta _{1}\mu _{2}}s_{\lambda _{2}}s_{\mu
_{2}^{\ast }}=s_{\alpha _{1}^{\ast }}as_{\beta _{1}}\in I_{0}
\end{equation*}%
and%
\begin{equation*}
\sum_{\left( \alpha _{1}\lambda _{2},\beta _{1}\mu _{2}\right) \in
F}r_{\alpha _{1}\lambda _{2},\beta _{1}\mu _{2}}s_{\alpha _{1}\lambda
_{2}}s_{\left( \beta _{1}\mu _{2}\right) ^{\ast }}=s_{\alpha _{1}}\left(
s_{\alpha _{1}^{\ast }}as_{\beta _{1}}\right) s_{\beta _{1}^{\ast }}\in
{\normalsize \operatorname{KP}}_{R}\left( \Lambda \right) _{n_{1}}I_{0}{\normalsize
\operatorname{KP}}_{R}\left( \Lambda \right) _{n_{2}}\text{.}
\end{equation*}%
Therefore,%
\begin{align*}
a& =\sum_{\left( \lambda _{1}\lambda _{2},\mu _{1}\mu _{2}\right) \in
F}r_{\lambda _{1}\lambda _{2},\mu _{1}\mu _{2}}s_{\lambda _{1}\lambda
_{2}}s_{\left( \mu _{1}\mu _{2}\right) ^{\ast }} \\
& =\sum_{\left\{ \left( \alpha _{1},\beta _{1}\right) :\left( \alpha
_{1}\alpha _{2},\beta _{1}\beta _{2}\right) \in F\right\} }\sum_{\left(
\alpha _{1}\lambda _{2},\beta _{1}\mu _{2}\right) \in F}r_{\alpha
_{1}\lambda _{2},\beta _{1}\mu _{2}}s_{\alpha _{1}\lambda _{2}}s_{\left(
\beta _{1}\mu _{2}\right) ^{\ast }}
\end{align*}%
belongs to ${\normalsize \operatorname{KP}}_{R}\left( \Lambda \right) _{n_{1}}I_{0}%
{\normalsize \operatorname{KP}}_{R}\left( \Lambda \right) _{n_{2}}$, and $%
I_{n}\subseteq {\normalsize \operatorname{KP}}_{R}\left( \Lambda \right)
_{n_{1}}I_{0}{\normalsize \operatorname{KP}}_{R}\left( \Lambda \right) _{n_{2}}$.

Now since $I$ is a graded ideal $I=\bigoplus_{n\in \mathbb{Z}^{k}}I_{n}$ and
$I$ is generated as an ideal by $I_{0}$.
\end{proof}

\begin{proof}[Proof of Theorem \protect\ref{the-graded-uniqueness-theorem}]
Because $\pi$ is graded, we have that $\ker \pi$ is a graded ideal. By Lemma %
\ref{generator-of-ideal}, the ideal $\ker \pi $ is generated by the set $%
\ker \pi \cap {\normalsize \operatorname{KP}}_{R}\left( \Lambda \right) _{0}$. Thus
it suffices to show $\pi |_{{\normalsize \operatorname{KP}}_{R}\left( \Lambda
\right) _{0}}:{\normalsize \operatorname{KP}}_{R}\left( \Lambda \right)
_{0}\rightarrow A$ is injective. However, the injectivity follows from
Theorem \ref{injectivity-on-the core}.
\end{proof}

One immediate application of Theorem \ref{the-graded-uniqueness-theorem} is:

\begin{proposition}
\label{KP-is-dense-in-C}Let $\Lambda $ be a finitely aligned $k$-graph. Let $%
\left\{ s_{\lambda },s_{\mu ^{\ast }}:\lambda ,u\in \Lambda \right\} $ be
the universal Kumjian-Pask $\Lambda $-family for $R=\mathbb{C}$ and $\left\{
t_{\lambda }:\lambda \in \Lambda \right\} $ be the universal Cuntz-Krieger $%
\Lambda $-family. Then there is an isomorphism $\pi _{t}:{\normalsize \operatorname{%
KP}}_{\mathbb{C}}\left( \Lambda \right) \rightarrow \operatorname{span}_{\mathbb{C}%
}\left\{ t_{\lambda }t_{\mu }^{\ast }:\lambda ,\mu \in \Lambda \right\} $
such that $\pi _{t}\left( s_{\lambda }\right) =t_{\lambda }$ and $\pi
_{t}\left( s_{\mu ^{\ast }}\right) =t_{\mu }^{\ast }$ for $\lambda ,u\in
\Lambda $. In particular, ${\normalsize \operatorname{KP}}_{\mathbb{C}}\left(
\Lambda \right) $ is isomorphic to a dense subalgebra of $C^{\ast }\left(
\Lambda \right) $.
\end{proposition}

\begin{proof}
Since $\left\{ t_{\lambda }:\lambda \in \Lambda \right\} $ satisfies
(TCK1-3) and (CK), then $\left\{ t_{\lambda },t_{\mu }^{\ast }:\lambda ,\mu
\in \Lambda \right\} $ also satisfies (KP1-4) and is a Kumjian-Pask $\Lambda
$-family in $C^{\ast }\left( \Lambda \right) $. Thus the universal property
of ${\normalsize \operatorname{KP}}_{\mathbb{C}}\left( \Lambda \right) $ gives a
homomorphism $\pi _{t}$ from ${\normalsize \operatorname{KP}}_{\mathbb{C}}\left(
\Lambda \right) $ onto the dense subalgebra%
\begin{equation*}
A:=\operatorname{span}_{\mathbb{C}}\left\{ t_{\lambda }t_{\mu }^{\ast }:\lambda ,\mu
\in \Lambda \right\}
\end{equation*}%
of $C^{\ast }\left( \Lambda \right) $.

Next we show the injectivity of $\pi _{t}$. By Theorem \ref%
{the-graded-uniqueness-theorem}, it suffices to show that $\pi _{t}$ is a $%
\mathbb{Z}^{k}$-graded algebra homomorphism. We claim that $A$ is graded by%
\begin{equation*}
A_{n}:=\operatorname{span}_{\mathbb{C}}\left\{ t_{\lambda }t_{\mu }^{\ast }:\lambda
,\mu \in \Lambda ,d\left( \lambda \right) -d\left( \mu \right) =n\right\}
\text{.}
\end{equation*}%
Note that for $\lambda ,\mu ,\rho ,\tau \in \Lambda $ with $d\left( \lambda
\right) -d\left( \mu \right) =n$ and $d\left( \rho \right) -d\left( \tau
\right) =m$, we have%
\begin{align*}
t_{\lambda }t_{\mu }^{\ast }t_{\rho }t_{\tau }^{\ast }& =t_{\lambda }\big(%
\sum_{(\mu ^{\prime },\rho ^{\prime })\in \Lambda ^{\min }\left( \mu ,\rho
\right) }t_{\mu ^{\prime }}t_{\rho ^{\prime }}^{\ast }\big)t_{\tau }^{\ast }%
\text{ (by (TCK3))} \\
& =\sum_{(\mu ^{\prime },\rho ^{\prime })\in \Lambda ^{\min }\left( \mu
,\rho \right) }t_{\lambda \mu ^{\prime }}t_{\tau \rho ^{\prime }}^{\ast }
\end{align*}%
and for $(\mu ^{\prime },\rho ^{\prime })\in \Lambda ^{\min }\left( \mu
,\rho \right) $,%
\begin{align*}
d\left( \lambda \mu ^{\prime }\right) -d\left( \tau \rho ^{\prime }\right) &
=d\left( \lambda \right) +d\left( \mu ^{\prime }\right) -d\left( \tau
\right) -d\left( \rho ^{\prime }\right) \\
& =d\left( \lambda \right) +\left( d\left( \mu \right) \vee d\left( \rho
\right) -d\left( \mu \right) \right) \\
& \quad -d\left( \tau \right) -\left( d\left( \mu \right) \vee d\left( \rho
\right) -d\left( \rho \right) \right) \\
& =\left( d\left( \lambda \right) -d\left( \mu \right) \right) -\left(
d\left( \tau \right) -d\left( \rho \right) \right) \\
& =n+m\text{.}
\end{align*}%
Hence $A_{n}A_{m}\subseteq A_{n+m}$. Since each spanning element $t_{\lambda
}t_{\mu }^{\ast }$ belongs to $A_{d\left( \lambda \right) -d\left( \mu
\right) }$, every element $a$ of $A$ can be written as a finite sum $\sum
a_{n}$ with $a_{n}\in A_{n}$. For $a_{n}\in A_{n}$ such that a finite sum $%
\sum a_{n}=0$, then we have each $a_{n}=0$ by following the argument of \cite%
[Lemma 7.4]{ACaHR13}. Thus $\left\{ A_{n}:n\in \mathbb{Z}^{k}\right\} $ is a
grading of $A$, as claimed. Then $\pi _{t}$ is a $\mathbb{Z}^{k}$-graded and
by Theorem \ref{the-graded-uniqueness-theorem}, $\pi _{t}$ is injective.
\end{proof}

\section{Steinberg algebras}

\label{Section-Steinberg-algebra}Steinberg algebras were introduced by
Steinberg in \cite{St10} and are algebraic analogues of groupoid $C^{\ast}$%
-algebras. In \cite{CS15}, Clark and Sims show that for every $1$-graph $E$,
its Leavitt path algebra is isomorphic to a Steinberg algebra. In this
section, we show that for every finitely aligned $k$-graph $\Lambda $, its
Kumjian-Pask algebra is isomorphic to a Steinberg algebra (Proposition \ref%
{KP-is-isomorphic-to-Steinberg-algebras}). We start out with an introduction
to groupoids and Steinberg algebras in general.

A groupoid $\mathcal{G}$ is a small category in which every morphism has an
inverse. For a groupoid $\mathcal{G}$, we write $r\left( a\right) $ and $%
s\left( a\right) $ to denote the \emph{range} and \emph{source} of $a\in
\mathcal{G}$. Because $r\left( a\right) =s\left( a^{-1}\right) $ for $a\in
\mathcal{G}$, then $r$ and $s$ have the common image. We call this common
image the \emph{unit space }of $\mathcal{G}$ and denote it $\mathcal{G}%
^{\left( 0\right) }$. A pair $\left( a,b\right) \in \mathcal{G}\times
\mathcal{G}$ is said \emph{composable }if $s\left( a\right) =r\left(
b\right) $. We then use notation $\mathcal{G}^{\left( 2\right) }$ to denote
the collection of composable pairs in $\mathcal{G}$. For $A,B\subseteq
\mathcal{G}$, we write%
\begin{equation*}
AB:=\left\{ ab:a\in A,b\in B,\left( a,b\right) \in \mathcal{G}^{\left(
2\right) }\right\} \text{.}
\end{equation*}

We say $\mathcal{G}$ is a \emph{topological groupoid} if $\mathcal{G}$ is
endowed with a topology such that composition and inversion on $\mathcal{G}$
are continuous. We also call an open set $U\subseteq \mathcal{G}$ an \emph{%
open bisection }if $s$ and $r$ restricted to $U$ are homeomorhisms into $%
\mathcal{G}^{(0)}$. Finally, we call $\mathcal{G}$ \emph{ample} if $\mathcal{%
G}$ has a basis of compact open bisections.

\begin{remark}
\label{remark-of-groupouid-Glambda-totally-disconnected} Note that if $%
\mathcal{G}$ is ample, then $\mathcal{G}$ is locally compact and \'{e}tale.
In fact, $\mathcal{G}$ is Hausdorff ample if and only if $\mathcal{G}$ is
locally compact, Hausdorff and \'etale with totally disconnected unit space.
\end{remark}

Now suppose that $\mathcal{G}$ is a Hausdorff ample groupoid and $R$ is a
commutative ring with $1$. As in \cite[Section 2.2]{CE-M15}, the Steinberg
algebra\footnote{%
In \cite{St10}, Steinberg writes $R\mathcal{G}$ to denote $A_{R}(\mathcal{G}%
) $.} associated to $\mathcal{G}$ is
\begin{equation*}
A_{R}\left( \mathcal{G}\right) :=\{f:\mathcal{G}\rightarrow R:f\text{ is
locally constant and has compact support}\}
\end{equation*}%
where addition and scalar multiplication are defined pointwise, and
convolution is given by%
\begin{equation*}
\left( f\star g\right) \left( a\right) :=\sum_{r\left( a\right) =r\left(
b\right) }f\left( b\right) g\left( b^{-1}a\right) \text{.}
\end{equation*}%
Furthermore, for compact open bisections $U$ and $V$, we have the
characteristic function $1_U \in A_R(\mathcal{G})$ and
\begin{equation*}
1_{U}\star 1_{V}=1_{UV}
\end{equation*}%
\cite[Proposition 4.3]{St10}. Note that for $f\in A_{R}\left( \mathcal{G}%
\right) $, $\operatorname{supp}\left( f\right) $ is clopen (\cite[Remark 2.1]{CE-M15}%
).

\begin{example}
\label{groupouid-Glambda}To each finitely aligned $k$-graph $\Lambda $, we
define the associated \emph{boundary-path groupoid} $\mathcal{G}_{\Lambda }$
from \cite[Definition 4.8]{Y07} as follows. Write%
\begin{equation*}
\Lambda \ast _{s}\Lambda :=\left\{ \left( \lambda ,\mu \right) \in \Lambda
\times \Lambda :s\left( \lambda \right) =s\left( \mu \right) \right\} \text{.%
}
\end{equation*}%
The objects of $\mathcal{G}_{\Lambda}$ are
\begin{equation*}
\operatorname{Obj}\left( \mathcal{G}_{\Lambda }\right) :=\partial \Lambda \text{.}
\end{equation*}%
The morphisms are
\begin{align*}
\operatorname{Mor}\left( \mathcal{G}_{\Lambda }\right) & :=\{\left( \lambda
z,d\left( \lambda \right) -d\left( \mu \right) ,\mu z\right) \in \partial
\Lambda \times \mathbb{Z}^{k}\times \partial \Lambda : \\
& \quad \quad \quad \quad \quad \quad \left( \lambda ,\mu \right) \in
\Lambda \ast _{s}\Lambda ,z\in s\left( \lambda \right) \partial \Lambda \} \\
& =\{\left( x,m,y\right) \in \partial \Lambda \times \mathbb{Z}^{k}\times
\partial \Lambda :\text{there exists }p,q\in \mathbb{N}^{k}\text{ such that}
\\
& \quad \quad \quad \quad \quad \quad p\leq d\left( x\right) ,q\leq d\left(
y\right) ,p-q=m\text{ and }\sigma ^{p}x=\sigma ^{q}y\}\text{.}
\end{align*}%
The range and source maps are given by $r\left( x,m,y\right) :=x$ and $%
s\left( x,m,y\right) :=y$, and composition is defined such that%
\begin{equation*}
\left( \left( x_{1},m_{1},y_{1}\right) ,\left( y_{1},m_{2},y_{2}\right)
\right) \mapsto \left( x_{1},m_{1}+m_{2},y_{2}\right) \text{.}
\end{equation*}%
Fianlly inversion is given by $\left( x,m,y\right) \mapsto \left(
y,-m,x\right) $.

Next, we show how to realise $\mathcal{G}_{\Lambda}$ as a topological
groupoid. For $\left( \lambda ,\mu \right) \in \Lambda \ast _{s}\Lambda $
and finite non-exhaustive subset $G\subseteq s\left( \lambda \right) \Lambda
$, we write%
\begin{equation*}
Z_{\Lambda }\left( \lambda \right) :=\lambda \partial \Lambda \text{,}
\end{equation*}%
\begin{equation*}
Z_{\Lambda }\left( \left. \lambda \right\backslash G\right) :=\left.
Z_{\Lambda }\left( \lambda \right) \right\backslash \Big(\bigcup_{\nu \in
G}Z_{\Lambda }\left( \lambda \nu \right) \Big)\text{,}
\end{equation*}%
\begin{align*}
Z_{\Lambda }\left( \lambda \ast _{s}\mu \right) & :=\{\left( x,d\left(
\lambda \right) -d\left( \mu \right) ,y\right) \in \mathcal{G}_{\Lambda
}:x\in Z_{\Lambda }\left( \lambda \right) ,y\in Z_{\Lambda }\left( \mu
\right) \\
& \quad \quad \quad \quad \quad \text{ and }\sigma ^{d\left( \lambda \right)
}x=\sigma ^{d\left( \mu \right) }y\}\text{,}
\end{align*}%
and%
\begin{equation*}
Z_{\Lambda }\left( \left. \lambda \ast _{s}\mu \right\backslash G\right)
:=\left. Z_{\Lambda }\left( \lambda \ast _{s}\mu \right) \right\backslash %
\Big(\bigcup_{\nu \in G}Z_{\Lambda }\left( \lambda \nu \ast _{s}\mu \nu
\right) \Big).
\end{equation*}%
The sets $Z_{\Lambda }\left( \left. \lambda \ast _{s}\mu \right\backslash
G\right) $ form a basis of compact open bisections for a second-countable,
Hausdorff topology on $\mathcal{G}_{\Lambda }$ under which it is an ample
groupoid. Further, the sets $Z_{\Lambda }\left( \left. \lambda
\right\backslash G\right) $ form a basis of compact open sets for $\mathcal{G%
}_{\Lambda }^{\left( 0\right) }$.
\end{example}

\begin{remark}
\label{remark-of-groupouid-Glambda}A number of notes of this example:

\begin{enumerate}
\item[(i)] We think of $\mathcal{G}_{\Lambda }^{\left( 0\right) }=\partial
\Lambda $ as a subset of $\mathcal{G}_{\Lambda }$ under the correspondence $%
x\mapsto \left( x,0,x\right) $.

\item[(ii)] In \cite{Y07}, Yeend defines $Z_{\Lambda }\left( \left. \lambda
\right\backslash G\right) $ and $Z_{\Lambda }\left( \left. \lambda \ast
_{s}\mu \right\backslash G\right) $ where $G$ is finite. However, if $G$ is
exhaustive, then $Z_{\Lambda }\left( \left. \lambda \right\backslash
G\right) $ and $Z_{\Lambda }\left( \left. \lambda \ast _{s}\mu
\right\backslash G\right) $ are empty sets. Thus our definitions make sure
that both $Z_{\Lambda }\left( \left. \lambda \right\backslash G\right) $ and
$Z_{\Lambda }\left( \left. \lambda \ast _{s}\mu \right\backslash G\right) $
are non-empty.
\end{enumerate}
\end{remark}

Next we generalise \cite[Proposition 4.3]{CFST14} as follows:

\begin{proposition}
\label{KP-is-isomorphic-to-Steinberg-algebras}Let $\Lambda $ be a finitely
aligned $k$-graph and $\mathcal{G}_{\Lambda }$ be its boundary-path groupoid
as defined in Example \ref{groupouid-Glambda}. Let $R$ be a commutative ring
with $1$. Then there is an isomorphism $\pi _{T}:{\normalsize \operatorname{KP}}%
_{R}\left( \Lambda \right) \rightarrow A_{R}\left( \mathcal{G}_{\Lambda
}\right) $ such that $\pi _{T}\left( s_{\lambda }\right) =1_{Z_{\Lambda
}\left( \lambda \ast _{s}s\left( \lambda \right) \right) }$ and $\pi
_{T}\left( s_{\mu ^{\ast }}\right) =1_{Z_{\Lambda }\left( s\left( \mu
\right) \ast _{s}\mu \right) }$ for $\lambda ,u\in \Lambda $.
\end{proposition}

The only part of the proof of Proposition~\ref%
{KP-is-isomorphic-to-Steinberg-algebras} that requires much additional work
is showing the surjectivity of $\pi _{T}$. For this, we establish the
following two lemmas. These lemmas show that the characteristic function
associated to a compact open set in $\mathcal{G}_{\Lambda}$ can be written
as a sum of elements in the form $1_{Z_{\Lambda }\left( \left. \lambda \ast
_{s}\mu \right\backslash G\right) }$.

\begin{lemma}
\label{intersection-of-Z}Let $\left( \lambda ,\mu \right) ,\left( \lambda
^{\prime },\mu ^{\prime }\right) \in \Lambda \ast _{s}\Lambda $, $G\subseteq
s\left( \lambda \right) \Lambda $, and $G^{\prime }\subseteq s\left( \lambda
^{\prime }\right) \Lambda $. Define $F:=\Lambda ^{\min }\left( \lambda
,\lambda ^{\prime }\right) \cap \Lambda ^{\min }\left( \mu ,\mu ^{\prime
}\right) $. Then
\begin{equation}
Z_{\Lambda }\left( \left. \lambda \ast _{s}\mu \right\backslash G\right)
\cap Z_{\Lambda }\left( \left. \lambda ^{\prime }\ast _{s}\mu ^{\prime
}\right\backslash G^{\prime }\right) =\bigsqcup\limits_{(\gamma ,\gamma
^{\prime })\in F}Z_{\Lambda }\left( \left. \lambda \gamma \ast _{s}\mu
^{\prime }\gamma ^{\prime }\right\backslash \left[ \operatorname{Ext}\left( \gamma
;G\right) \cup \operatorname{Ext}\left( \gamma ^{\prime };G^{\prime }\right) \right]
\right) \text{.}  \tag{*}
\end{equation}
\end{lemma}

\begin{proof}
We generalise the argument of \cite[Example 3.2]{CS15} for $1$-graphs. First
we show that the collection%
\begin{equation*}
\left\{ Z_{\Lambda }\left( \left. \lambda \gamma \ast _{s}\mu ^{\prime
}\gamma ^{\prime }\right\backslash \left[ \operatorname{Ext}\left( \gamma ;G\right)
\cup \operatorname{Ext}\left( \gamma ^{\prime };G^{\prime }\right) \right] \right)
:(\gamma ,\gamma ^{\prime })\in F\right\}
\end{equation*}%
is disjoint. It suffices to show that the collection%
\begin{equation*}
\left\{ Z_{\Lambda }\left( \lambda \gamma \ast _{s}\mu ^{\prime }\gamma
^{\prime }\right) :(\gamma ,\gamma ^{\prime })\in F\right\}
\end{equation*}%
is disjoint. Suppose for contradiction that there exist $(\gamma ,\gamma
^{\prime }),(\gamma ^{\prime \prime },\gamma ^{\prime \prime \prime })\in F$
such that $(\gamma ,\gamma ^{\prime })\neq (\gamma ^{\prime \prime },\gamma
^{\prime \prime \prime })$ and $V:=Z_{\Lambda }\left( \lambda \gamma \ast
_{s}\mu ^{\prime }\gamma ^{\prime }\right) \cap Z_{\Lambda }\left( \lambda
\gamma ^{\prime \prime }\ast _{s}\mu ^{\prime }\gamma ^{\prime \prime \prime
}\right) \neq \emptyset $. Note that if $\gamma =\gamma ^{\prime \prime }$,
then%
\begin{align*}
\lambda ^{\prime }\gamma ^{\prime }& =\lambda \gamma \text{ (since }(\gamma
,\gamma ^{\prime })\in \Lambda ^{\min }\left( \lambda ,\lambda ^{\prime
}\right) \text{)} \\
& =\lambda \gamma ^{\prime \prime }\text{ (since }\gamma =\gamma ^{\prime
\prime }\text{)} \\
& =\lambda ^{\prime }\gamma ^{\prime \prime \prime }\text{ (since }(\gamma
^{\prime \prime },\gamma ^{\prime \prime \prime })\in \Lambda ^{\min }\left(
\lambda ,\lambda ^{\prime }\right) \text{)}
\end{align*}%
and $\gamma ^{\prime }=\gamma ^{\prime \prime \prime }$ by the factorisation
property, which contradicts $(\gamma ,\gamma ^{\prime })\neq (\gamma
^{\prime \prime },\gamma ^{\prime \prime \prime })$. The same argument shows
that $\gamma ^{\prime }=\gamma ^{\prime \prime \prime }$ implies $\gamma
=\gamma ^{\prime \prime }$. Hence $\gamma \neq \gamma ^{\prime \prime }$ and
$\gamma ^{\prime }\neq \gamma ^{\prime \prime \prime }$. Meanwhile, since $%
(\gamma ,\gamma ^{\prime }),(\gamma ^{\prime \prime },\gamma ^{\prime \prime
\prime })\in F$, then $d\left( \gamma \right) =d\left( \gamma ^{\prime
\prime }\right) $ and $d\left( \gamma ^{\prime }\right) =d\left( \gamma
^{\prime \prime \prime }\right) $. Take $\left( x,m,y\right) \in V$. Then $%
x\in Z_{\Lambda }\left( \lambda \gamma \right) $ and $x\in Z_{\Lambda
}\left( \lambda \gamma ^{\prime \prime }\right) $. Since $d\left( \gamma
\right) =d\left( \gamma ^{\prime \prime }\right) $, then $d\left( \lambda
\gamma \right) =d\left( \lambda \gamma ^{\prime \prime }\right) $ and $%
\gamma =x\left( d\left( \lambda \right) ,d\left( \lambda \gamma \right)
\right) =x\left( d\left( \lambda \right) ,d\left( \lambda \gamma ^{\prime
\prime }\right) \right) =\gamma ^{\prime \prime }$, which contradicts $%
\gamma \neq \gamma ^{\prime \prime }$. Hence the collection $\left\{
Z_{\Lambda }\left( \lambda \gamma \ast _{s}\mu ^{\prime }\gamma ^{\prime
}\right) :(\gamma ,\gamma ^{\prime })\in F\right\} $ is disjoint, and so is%
\begin{equation*}
\left\{ Z_{\Lambda }\left( \left. \lambda \gamma \ast _{s}\mu ^{\prime
}\gamma ^{\prime }\right\backslash \left[ \operatorname{Ext}\left( \gamma ;G\right)
\cup \operatorname{Ext}\left( \gamma ^{\prime };G^{\prime }\right) \right] \right)
:(\gamma ,\gamma ^{\prime })\in F\right\} \text{.}
\end{equation*}

Next we show the right inclusion of (*). Write%
\begin{equation*}
U:=Z_{\Lambda }\left( \left. \lambda \ast _{s}\mu \right\backslash G\right)
\cap Z_{\Lambda }\left( \left. \lambda ^{\prime }\ast _{s}\mu ^{\prime
}\right\backslash G^{\prime }\right)
\end{equation*}%
and take $\left( x,m,y\right) \in U$. We show $\left( x,m,y\right) \in
Z_{\Lambda }\left( \left. \lambda \gamma \ast _{s}\mu ^{\prime }\gamma
^{\prime }\right\backslash \left[ \operatorname{Ext}\left( \gamma ;G\right) \cup
\operatorname{Ext}\left( \gamma ^{\prime };G^{\prime }\right) \right] \right) $ for
some $\left( \gamma ,\gamma ^{\prime }\right) \in F$. Because $x\in
Z_{\Lambda }\left( \lambda \right) $ and $x\in Z_{\Lambda }\left( \lambda
^{\prime }\right) $, then $d\left( x\right) \geq d\left( \lambda \right)
\vee d\left( \lambda ^{\prime }\right) $ and there exists $\left( \gamma
,\gamma ^{\prime }\right) \in \Lambda ^{\min }\left( \lambda ,\lambda
^{\prime }\right) $ such that
\begin{equation}
x\in Z_{\Lambda }\left( \lambda \gamma \right) .
\label{equ-x-in-Z(lambda,gamma)}
\end{equation}%
Using a similar argument, there exists $\left( \gamma ^{\prime \prime
},\gamma ^{\prime \prime \prime }\right) \in \Lambda ^{\min }\left( \mu ,\mu
^{\prime }\right) $ such that
\begin{equation}
y\in Z_{\Lambda }\left( \mu \gamma ^{\prime \prime }\right) .
\label{equ-y-in-Z(mu,gamma)}
\end{equation}

We claim that $\gamma =\gamma ^{\prime \prime }$ and $\gamma ^{\prime
}=\gamma ^{\prime \prime \prime }$. To see this, note that $m=d\left(
\lambda \right) -d\left( \mu \right) =d\left( \lambda ^{\prime }\right)
-d\left( \mu ^{\prime }\right) $ and%
\begin{align*}
d\left( \gamma \right) & =d\left( \lambda \right) \vee d\left( \lambda
^{\prime }\right) -d\left( \lambda \right) =\left( d\left( \mu \right)
+m\right) \vee \left( d\left( \mu ^{\prime }\right) +m\right) -\left(
d\left( \mu \right) +m\right) \\
& =\left( d\left( \mu \right) \vee d\left( \mu ^{\prime }\right) \right)
+m-\left( d\left( \mu \right) +m\right) =d\left( \mu \right) \vee d\left(
\mu ^{\prime }\right) -d\left( \mu \right) =d\left( \gamma ^{\prime \prime
}\right) \text{.}
\end{align*}%
Since $\left( x,m,y\right) \in Z_{\Lambda }\left( \left. \lambda \ast
_{s}\mu \right\backslash G\right) $, then $\sigma ^{d\left( \lambda \right)
}x=\sigma ^{d\left( \mu \right) }y$ and
\begin{equation*}
\gamma =\left( \sigma ^{d\left( \lambda \right) }x\right) \left( 0,d\left(
\gamma \right) \right) =\left( \sigma ^{d\left( \mu \right) }y\right) \left(
0,d\left( \gamma ^{\prime }\right) \right) =\gamma ^{\prime \prime }\text{.}
\end{equation*}%
Using a similar argument, we also get $\gamma ^{\prime }=\gamma ^{\prime
\prime \prime }$ proving the claim.

Next we show that $\left( x,m,y\right) \in Z_{\Lambda }\left( \lambda \gamma
\ast _{s}\mu ^{\prime }\gamma ^{\prime }\right) $. By %
\eqref{equ-x-in-Z(lambda,gamma)} and \eqref{equ-y-in-Z(mu,gamma)}, we have $%
x\in Z_{\Lambda }\left( \lambda \gamma \right) $ and $y\in Z_{\Lambda
}\left( \mu \gamma ^{\prime \prime }\right) $. Since $\gamma =\gamma
^{\prime \prime }$, $\gamma ^{\prime }=\gamma ^{\prime \prime \prime }$, $%
\left( \gamma ^{\prime \prime },\gamma ^{\prime \prime \prime }\right) \in
\Lambda ^{\min }\left( \mu ,\mu ^{\prime }\right) $, then $\mu \gamma
^{\prime \prime }=\mu \gamma =\mu ^{\prime }\gamma ^{\prime }$ and $y\in
Z_{\Lambda }\left( \mu ^{\prime }\gamma ^{\prime }\right) .$ On the other
hand, since $\left( x,m,y\right) \in Z_{\Lambda }\left( \left. \lambda \ast
_{s}\mu \right\backslash G\right) $, then $\sigma ^{d\left( \lambda \right)
}x=\sigma ^{d\left( \mu \right) }y$ and%
\begin{equation*}
\sigma ^{d\left( \lambda \gamma \right) }x=\sigma ^{d\left( \mu \gamma
\right) }y=\sigma ^{d(\mu ^{\prime }\gamma ^{\prime })}y
\end{equation*}%
since $\mu \gamma =\mu ^{\prime }\gamma ^{\prime }$. Since $m=d\left(
\lambda \right) -d\left( \mu \right) =d\left( \lambda \gamma \right)
-d\left( \mu ^{\prime }\gamma ^{\prime }\right) $, then $\left( x,m,y\right)
\in Z_{\Lambda }\left( \lambda \gamma \ast _{s}\mu ^{\prime }\gamma ^{\prime
}\right) $, as required.

Finally we show that $\left( x,m,y\right) \notin Z_{\Lambda }\left( \lambda
\gamma \nu \ast _{s}\mu ^{\prime }\gamma ^{\prime }\nu \right) $ for all $%
\nu \in \operatorname{Ext}\left( \gamma ;G\right) \cup \operatorname{Ext}\left( \gamma
^{\prime };G^{\prime }\right) $. Suppose for a contradiction that there
exists $\nu \in \operatorname{Ext}\left( \gamma ;G\right) \cup \operatorname{Ext}\left(
\gamma ^{\prime };G^{\prime }\right) $ such that $\left( x,m,y\right) \in
Z_{\Lambda }\left( \lambda \gamma \nu \ast _{s}\mu ^{\prime }\gamma ^{\prime
}\nu \right) $. Without loss of generality, suppose $\nu \in \operatorname{Ext}%
\left( \gamma ;G\right) $. Then there exists $\nu ^{\prime }\in G$ such that
$\gamma \nu \in Z_{\Lambda }\left( \nu ^{\prime }\right) $. Since $x\in
Z_{\Lambda }\left( \lambda \gamma \nu \right) $, $y\in Z_{\Lambda }\left(
\mu ^{\prime }\gamma ^{\prime }\nu \right) =Z_{\Lambda }\left( \mu \gamma
\nu \right) $, and $\gamma \nu \in Z_{\Lambda }\left( \nu ^{\prime }\right) $%
, then $x\in Z_{\Lambda }\left( \lambda \nu ^{\prime }\right) $ and $y\in
Z_{\Lambda }\left( \mu \nu ^{\prime }\right) $ where $\nu ^{\prime }\in G$.
This contradicts $\left( x,m,y\right) \in Z_{\Lambda }\left( \left. \lambda
\ast _{s}\mu \right\backslash G\right) $. Hence
\begin{equation*}
\left( x,m,y\right) \in Z_{\Lambda }\left( \left. \lambda \gamma \ast
_{s}\mu ^{\prime }\gamma ^{\prime }\right\backslash \left[ \operatorname{Ext}\left(
\gamma ;G\right) \cup \operatorname{Ext}\left( \gamma ^{\prime };G^{\prime }\right) %
\right] \right)
\end{equation*}%
and%
\begin{equation*}
U\subseteq \bigsqcup\limits_{(\gamma ,\gamma ^{\prime })\in F}Z_{\Lambda
}\left( \left. \lambda \gamma \ast _{s}\mu ^{\prime }\gamma ^{\prime
}\right\backslash \left[ \operatorname{Ext}\left( \gamma ;G\right) \cup \operatorname{Ext}%
\left( \gamma ^{\prime };G^{\prime }\right) \right] \right) \text{.}
\end{equation*}

Next we show the left inclusion of (*). Take $\left( \gamma ,\gamma ^{\prime
}\right) \in F$ and
\begin{equation}
\left( x,m,y\right) \in Z_{\Lambda }\left( \left. \lambda \gamma \ast
_{s}\mu ^{\prime }\gamma ^{\prime }\right\backslash \left[ \operatorname{Ext}\left(
\gamma ;G\right) \cup \operatorname{Ext}\left( \gamma ^{\prime };G^{\prime }\right) %
\right] \right) \text{.}  \label{equ-x-in-Z-without-ext}
\end{equation}%
We show $\left( x,m,y\right) $ belongs to both $Z_{\Lambda }\left( \left.
\lambda \ast _{s}\mu \right\backslash G\right) $ and $Z_{\Lambda }\left(
\left. \lambda ^{\prime }\ast _{s}\mu ^{\prime }\right\backslash G^{\prime
}\right) $. Without loss of generality, it suffices to show $\left(
x,m,y\right) \in Z_{\Lambda }\left( \left. \lambda \ast _{s}\mu
\right\backslash G\right) $. First we show that $\left( x,m,y\right) \in
Z_{\Lambda }\left( \lambda \ast _{s}\mu \right) $. Note that we have $\mu
\gamma =\mu ^{\prime }\gamma ^{\prime }$ and $m=d\left( \lambda \gamma
\right) -d\left( \mu ^{\prime }\gamma ^{\prime }\right) =d\left( \lambda
\right) -d\left( \mu \right) $. On the other hand, $\left( x,m,y\right) \in
Z_{\Lambda }\left( \lambda \gamma \ast _{s}\mu ^{\prime }\gamma ^{\prime
}\right) $ also implies $x\in Z_{\Lambda }\left( \lambda \gamma \right) $
and $y\in Z_{\Lambda }\left( \mu ^{\prime }\gamma ^{\prime }\right)
=Z_{\Lambda }\left( \mu \gamma \right) $. Furthermore,
\begin{align*}
\sigma ^{\left( \lambda \right) }x& =\left[ x\left( d\left( \lambda \right)
,d\left( \lambda \gamma \right) \right) \right] \left[ \sigma ^{\left(
\lambda \gamma \right) }x\right] \\
& =\gamma \left[ \sigma ^{\left( \lambda \gamma \right) }x\right] \text{
(since }x\left( d\left( \lambda \right) ,d\left( \lambda \gamma \right)
\right) =\gamma \text{)} \\
& =\gamma \lbrack \sigma ^{(\mu ^{\prime }\gamma ^{\prime })}y]\text{ (since
}\sigma ^{\left( \lambda \gamma \right) }x=\sigma ^{(\mu ^{\prime }\gamma
^{\prime })}y\text{)} \\
& =\left[ y\left( d\left( \mu \right) ,d\left( \mu \gamma \right) \right) %
\right] [\sigma ^{(\mu ^{\prime }\gamma ^{\prime })}y]\text{ (since }y\left(
d\left( \mu \right) ,d\left( \mu \gamma \right) \right) =\gamma \text{)} \\
& =\left[ y\left( d\left( \mu \right) ,d\left( \mu \gamma \right) \right) %
\right] [\sigma ^{(\mu \gamma )}y]\text{ (since }\mu \gamma =\mu ^{\prime
}\gamma ^{\prime }) \\
& =\sigma ^{(\mu )}y
\end{align*}%
and then $\left( x,m,y\right) \in Z_{\Lambda }\left( \lambda \ast _{s}\mu
\right) $, as required.

To complete the proof, we have to show $\left( x,m,y\right) \notin
Z_{\Lambda }\left( \lambda \nu \ast _{s}\mu \nu \right) $ for all $\nu \in G$%
. Suppose for contradiction that there exists $\nu \in G$ such that $\left(
x,m,y\right) \in Z_{\Lambda }\left( \lambda \nu \ast _{s}\mu \nu \right) $.
In particular, $x\in Z_{\Lambda }\left( \lambda \nu \right) $. Since $x\in
Z_{\Lambda }\left( \lambda \gamma \right) $ and $x\in Z_{\Lambda }\left(
\lambda \nu \right) $, then there exists $\nu ^{\prime }\in \operatorname{Ext}\left(
\gamma ;\left\{ \nu \right\} \right) $ such that $x\in Z_{\Lambda }\left(
\lambda \gamma \nu ^{\prime }\right) $. Hence%
\begin{align*}
\sigma ^{(\lambda \gamma \nu ^{\prime })}x& =\sigma ^{(\mu \gamma \nu
^{\prime })}y\text{ (since }\sigma ^{\left( \lambda \right) }x=\sigma ^{(\mu
)}y\text{)} \\
& =\sigma ^{(\mu ^{\prime }\gamma ^{\prime }\nu ^{\prime })}y\text{ (since }%
\mu \gamma =\mu ^{\prime }\gamma ^{\prime }\text{),}
\end{align*}%
\begin{align}
\left( \sigma ^{(\mu )}y\right) \left( 0,d\left( \gamma \nu ^{\prime
}\right) \right) & =\left( \sigma ^{\left( \lambda \right) }x\right) \left(
0,d\left( \gamma \nu ^{\prime }\right) \right) \text{ (since }\sigma
^{\left( \lambda \right) }x=\sigma ^{(\mu )}y\text{)}
\label{equ-intersection-of-Z-1} \\
& =x\left( d\left( \lambda \right) ,d\left( \lambda \gamma \nu ^{\prime
}\right) \right)  \notag \\
& =\gamma \nu ^{\prime }\text{ (since }x\in Z_{\Lambda }\left( \lambda
\gamma \nu ^{\prime }\right) \text{),}  \notag
\end{align}%
and%
\begin{align*}
y\left( 0,d\left( \mu ^{\prime }\gamma ^{\prime }\nu ^{\prime }\right)
\right) & =y\left( 0,d\left( \mu \gamma \nu ^{\prime }\right) \right) \text{
(since }\mu \gamma =\mu ^{\prime }\gamma ^{\prime }\text{)} \\
& =\mu \gamma \nu ^{\prime }\text{ (by \eqref{equ-intersection-of-Z-1})} \\
& =\mu ^{\prime }\gamma ^{\prime }\nu ^{\prime }\text{ (since }\mu \gamma
=\mu ^{\prime }\gamma ^{\prime }\text{).}
\end{align*}%
Furthermore,%
\begin{align*}
d\left( \lambda \gamma \nu ^{\prime }\right) -d\left( \mu ^{\prime }\gamma
^{\prime }\nu ^{\prime }\right) & =d\left( \lambda \gamma \right) -d\left(
\mu ^{\prime }\gamma ^{\prime }\right) \\
& =d\left( \lambda \gamma \right) -d\left( \mu \gamma \right) \text{ (since }%
\mu \gamma =\mu ^{\prime }\gamma ^{\prime }\text{)} \\
& =d\left( \lambda \right) -d\left( \mu \right) =m\text{.}
\end{align*}%
Hence $\left( x,m,y\right) \in Z_{\Lambda }\left( \lambda \gamma \nu
^{\prime }\ast _{s}\mu ^{\prime }\gamma ^{\prime }\nu ^{\prime }\right) $
for some $\nu ^{\prime }\in \operatorname{Ext}\left( \gamma ;\left\{ \nu \right\}
\right) \subseteq \operatorname{Ext}\left( \gamma ;G\right) $, which contradicts %
\eqref{equ-x-in-Z-without-ext}. The conclusion follows.
\end{proof}

\begin{lemma}
\label{compact-open-bisection-U-is-in-span}Let $\left\{ Z_{\Lambda }\left(
\left. \lambda _{i}\ast _{s}\mu _{i}\right\backslash G_{i}\right) \right\}
_{i=1}^{n}$ be a finite collection of compact open bisection sets and%
\begin{equation*}
U:=\bigcup_{i=1}^{n}Z_{\Lambda }\left( \left. \lambda _{i}\ast _{s}\mu
_{i}\right\backslash G_{i}\right) \text{.}
\end{equation*}%
Then%
\begin{equation*}
1_{U}\in \operatorname{span}_{R}\left\{ 1_{Z_{\Lambda }\left( \left. \lambda \ast
_{s}\mu \right\backslash G\right) }:\left( \lambda ,\mu \right) \in \Lambda
\ast _{s}\Lambda ,G\subseteq s\left( \lambda \right) \Lambda \right\} \text{.%
}
\end{equation*}
\end{lemma}

\begin{proof}
It is trivial for $n=1$. Now let $n=2$ and $F:=\Lambda ^{\min }\left(
\lambda _{1},\lambda _{2}\right) \cap \Lambda ^{\min }\left( \mu _{1},\mu
_{2}\right) \,$. If $F=\emptyset $, then
\begin{equation*}
1_{U}=1_{Z_{\Lambda }\left( \left. \lambda \ast _{s}\mu \right\backslash
G\right) }+1_{Z_{\Lambda }\left( \left. \lambda ^{\prime }\ast _{s}\mu
^{\prime }\right\backslash G^{\prime }\right) }\text{.}
\end{equation*}%
Otherwise, by Proposition \ref{intersection-of-Z}, we have
\begin{equation*}
1_{U}=1_{Z_{\Lambda }\left( \left. \lambda \ast _{s}\mu \right\backslash
G\right) }+1_{Z_{\Lambda }\left( \left. \lambda ^{\prime }\ast
_{s}\mu^{\prime}\right\backslash G^{\prime }\right) }-\sum\limits_{(\gamma
,\gamma ^{\prime })\in F}1_{Z_{\gamma ,\gamma ^{\prime }}}
\end{equation*}%
where $Z_{\gamma ,\gamma ^{\prime }}:=Z_{\Lambda }\left( \left. \lambda
\gamma \ast _{s}\mu ^{\prime }\gamma ^{\prime }\right\backslash \operatorname{Ext}%
\left( \gamma ;G\right) \cup \operatorname{Ext}\left( \gamma ^{\prime };G^{\prime
}\right) \right) $, as required. For $n\geq 3$, by using the
inclusion-exclusion principle and de Morgan's law, $1_{U}$ can be written as
a sum of elements in the form $1_{Z_{\Lambda }\left( \left. \lambda \ast
_{s}\mu \right\backslash G\right) }$.
\end{proof}

\begin{proof}[Proof of Proposition \protect\ref%
{KP-is-isomorphic-to-Steinberg-algebras}]
Define $T_{\lambda }:=1_{Z_{\Lambda }\left( \lambda \ast _{s}s\left( \lambda
\right) \right) }$. Then by \cite[Theorem 6.13]{FMY05} (or \cite[Example 7.1]%
{Y07}), $\left\{ T_{\lambda},T_{\mu ^{\ast }}:\lambda ,u\in \Lambda \right\}
$ is a Kumjian-Pask $\Lambda $-family in $A_{R}\left( \mathcal{G}_{\Lambda
}\right) $. Hence, there exists a homomorphism $\pi _{T}:{\normalsize \operatorname{%
KP}}_{R}\left( \Lambda \right) \rightarrow A_{R}\left( \mathcal{G}_{\Lambda
}\right) $ such that $\pi _{T}\left( s_{\lambda }\right) =T_{\lambda }$ and $%
\pi _{T}\left( s_{\mu ^{\ast }}\right) =T_{\mu ^{\ast }}$ for $\lambda ,\mu
\in \Lambda $ by Theorem~\ref{universal-KP-family}(a).

To see that $\pi _{T}$ is injective, first we show that $\pi _{T}$ is
graded. Take $\lambda ,\mu \in \Lambda $. Then $s_{\lambda }s_{\mu ^{\ast
}}\in {\normalsize \operatorname{KP}}_{R}\left( \Lambda \right) _{d\left( \lambda
\right) -d\left( \mu \right) }$ and%
\begin{equation*}
\pi _{T}\left( s_{\lambda }s_{\mu ^{\ast }}\right) =1_{Z_{\Lambda }\left(
\lambda \ast _{s}\mu \right) }=1_{\{\left( x,d\left( \lambda \right)
-d\left( \mu \right) ,y\right) :\left( \lambda ,\mu \right) \in \Lambda \ast
_{s}\Lambda ,z\in s\left( \lambda \right) \partial \Lambda \}}\in
A_{R}\left( \mathcal{G}_{\Lambda }\right) _{d\left( \lambda \right) -d\left(
\mu \right) }\text{.}
\end{equation*}%
Since for every $n\in \mathbb{Z}^{k}$, ${\normalsize \operatorname{KP}}_{R}\left(
\Lambda \right) _{n}$ is spanned by elements in the form $s_{\lambda }s_{\mu
^{\ast }}$ (Theorem \ref{universal-KP-family}.(c)), then for $n\in \mathbb{Z}%
^{k}$, $\pi _{T}\left( {\normalsize \operatorname{KP}}_{R}\left( \Lambda \right)
_{n}\right) \subseteq A_{R}\left( \mathcal{G}_{\Lambda }\right) _{n}$ and $%
\pi _{T}$ is graded. Since $\pi _{T}\left( rs_{v}\right) =r1_{Z_{\Lambda
}\left( v\ast _{s}v\right) }\neq 0$ for all $r\in \left. R\right\backslash
\left\{ 0\right\} $ and $v\in \Lambda ^{0}$, and $\pi _{T}$ is graded, then
by Theorem \ref{the-graded-uniqueness-theorem}, $\pi _{T}$ is injective, as
required.

Finally we show the surjectivity of $\pi _{T}$. Take $f\in A_{R}\left(
\mathcal{G}_{\Lambda }\right) $. By \cite[Lemma 2.2]{CE-M15}, $f$ can be
written as $\sum_{U\in F}a_{U}1_{U}$ where $a_{U}\in R$, each $U$ is in the
form $\bigcup_{i=1}^{n}Z_{\Lambda }\left( \left. \lambda _{i}\ast _{s}\mu
_{i}\right\backslash G_{i}\right) $ for some $n\in \mathbb{N}$, and $F$ is
finite set of mutually disjoint elements. Hence, to show $f\in \operatorname{im}%
\left( \pi _{T}\right) $, it suffices to show
\begin{equation*}
1_{U}\in \operatorname{im}\left( \pi _{T}\right)
\end{equation*}%
where $U:=\bigcup_{i=1}^{n}Z_{\Lambda }\left( \left. \lambda _{i}\ast
_{s}\mu _{i}\right\backslash G_{i}\right) $ for some $n\in \mathbb{N}$ and
collection $\left\{ Z_{\Lambda }\left( \left. \lambda _{i}\ast _{s}\mu
_{i}\right\backslash G_{i}\right) \right\} _{i=1}^{n}$. By Lemma \ref%
{compact-open-bisection-U-is-in-span}, $1_{U}$ can be written as the sum of
elements in the form $1_{Z_{\Lambda }\left( \left. \lambda \ast _{s}\mu
\right\backslash G\right) }$. On the other hand, for $\left( \lambda ,\mu
\right) \in \Lambda \ast _{s}\Lambda $ and finite $G\subseteq s\left(
\lambda \right) \Lambda $, we have%
\begin{align}
T_{\lambda }\big(\prod_{\nu \in G}\left( T_{s\left( \lambda \right) }-T_{\nu
}T_{\nu ^{\ast }}\right) \big)T_{\mu ^{\ast }}& =1_{Z_{\Lambda }\left(
\lambda \ast _{s}s\left( \lambda \right) \right) }\big(\prod_{\nu \in
G}\left( 1_{Z_{\Lambda }\left( s\left( \lambda \right) \ast _{s}s\left(
\lambda \right) \right) }-1_{Z_{\Lambda }\left( \nu \ast _{s}\nu \right)
}\right) \big)1_{Z_{\Lambda }\left( s\left( \mu \right) \ast _{s}\mu \right)
}  \label{equ-1_Z(lambda*mu-G)} \\
& =1_{Z_{\Lambda }\left( \lambda \ast _{s}s\left( \lambda \right) \right) }%
\big(\prod_{\nu \in G}\left( 1_{Z_{\Lambda }\left( \left. s\left( \lambda
\right) \ast _{s}s\left( \lambda \right) \right\backslash \left\{ \nu
\right\} \right) }\right) \big)1_{Z_{\Lambda }\left( s\left( \mu \right)
\ast _{s}\mu \right) }  \notag \\
& =1_{Z_{\Lambda }\left( \lambda \ast _{s}s\left( \lambda \right) \right) }%
\big(1_{\prod_{\nu \in G}Z_{\Lambda }\left( \left. s\left( \lambda \right)
\ast _{s}s\left( \lambda \right) \right\backslash \left\{ \nu \right\}
\right) }\big)1_{Z_{\Lambda }\left( s\left( \mu \right) \ast _{s}\mu \right)
}  \notag \\
& =1_{Z_{\Lambda }\left( \lambda \ast _{s}s\left( \lambda \right) \right) }%
\big(1_{Z_{\Lambda }\left( \left. s\left( \lambda \right) \ast _{s}s\left(
\lambda \right) \right\backslash G\right) }\big)1_{Z_{\Lambda }\left(
s\left( \mu \right) \ast _{s}\mu \right) }  \notag \\
& =1_{Z_{\Lambda }\left( \left. \lambda \ast _{s}\mu \right\backslash
G\right) }  \notag
\end{align}%
since $s\left( \lambda \right) =s\left( \mu \right) $. Hence, $1_{Z_{\Lambda
}\left( \left. \lambda \ast _{s}\mu \right\backslash G\right) }$ belongs to $%
\operatorname{im}\left( \pi _{T}\right) $ and then so does $1_{U}$, as required.
Therefore, $\pi _{T}$ is surjective and then is an isomorphism.
\end{proof}

\begin{remark}
\label{remark-directed-graph-and-locally-convex}Finitely aligned $k$-graphs
include $1$-graphs and row-finite $k$-graphs with no sources. Further, in
these cases, the boundary path groupoid $\mathcal{G}_{\Lambda }$ of Example %
\ref{groupouid-Glambda} coincides with $\mathcal{G}_{E}$ of \cite{CS15} and $%
\mathcal{G}_{\Lambda }$ of \cite{CFST14}. Thus, we have generalised Example
3.2 of \cite{CS15} and Proposition 4.3 of \cite{CFST14}. For locally convex
row-finite $k$-graphs, our construction gives a Steinberg algebra model of
the Kumjian-Pask algebras of \cite{CFaH14}.
\end{remark}

\section{Aperiodic higher-rank graphs and effective groupoids}

\label{Section-aperiodic-effective}In this section and Section \ref%
{Section-cofinal-minimal}, we investigate the relationship between a $k$%
-graph $\Lambda $ and its boundary-path groupoid $\mathcal{G}_{\Lambda }$ as
constructed in Example \ref{groupouid-Glambda}. We expect the Cuntz-Krieger
uniqueness theorem (Theorem \ref{the-CK-uniqueness-theorem}) to apply only
to \emph{aperiodic} finitely aligned $k$-graphs (definition below). On the
other hand, \emph{effective} groupoids (definition below) are needed in the
hypothesis of the Cuntz-Krieger uniqueness theorem for Steinberg algebras
(Theorem \ref{the-CK-uniqueness-theorem-for-Steinberg-algebras}). In this
section, our main result is Proposition \ref{aperiodic-iff-effective} which
says that a finitely aligned $k$-graph $\Lambda $ is aperiodic if and only
if the boundary-path groupoid $\mathcal{G}_{\Lambda }$ is effective.

We say a boundary path $x$ is \emph{aperiodic }if for all $\lambda ,\mu \in
\Lambda r\left( x\right) $, $\lambda \neq \mu $ implies $\lambda x\neq \mu x$%
. We say a finitely aligned $k$-graph $\Lambda $ is \emph{aperiodic }if for
each $v\in \Lambda ^{0}$, there exists an aperiodic boundary path $x$ with $%
r\left( x\right) =v$.

\begin{remark}
There are several equivalent ways to define the aperiodicity condition for
finitely aligned $k$-graphs (see \cite{FMY05,LS10,RSY04,Sh12}). However,
those definitions are equivalent by \cite[Proposition 3.6]{LS10} and \cite[%
Proposition 2.11]{Sh12}. The definition we use is called Condition (B$%
^{\prime }$) in \cite[Remark 7.3]{FMY05} and \cite[Definition 2.1.(ii)]{Sh12}%
.
\end{remark}

\begin{remark}
For $1$-graphs, the aperiodicity condition is known as Condition (L), which,
using our conventions, says that every cycle has an entry (see \cite%
{A15,AA08,BPRS00,KPR98,CBMS,T07,T11}).
\end{remark}

Next let $\mathcal{G}$ be a toplogical groupoid. Define $\operatorname{Iso}\left(
\mathcal{G}\right) $ the \emph{isotropy groupoid}\ of $\mathcal{G}$ by%
\begin{equation*}
\operatorname{Iso}\left( \mathcal{G}\right) :=\left\{ a\in \mathcal{G}:s\left(
a\right) =r\left( a\right) \right\} \text{.}
\end{equation*}%
We then say $\mathcal{G}$ is \emph{effective} if the interior of $\operatorname{Iso}%
\left( \mathcal{G}\right) $ is $\mathcal{G}^{\left( 0\right) }$. See \cite[%
Lemma 3.1]{BCFS14} for some equivalent characterisations.

\begin{proposition}
\label{aperiodic-iff-effective}Let $\Lambda $ be a finitely aligned $k$%
-graph. Then $\Lambda $ is aperiodic if and only if the boundary-path
groupoid $\mathcal{G}_{\Lambda }$ is effective.
\end{proposition}

\begin{proof}
$\left( \Rightarrow \right) $ First suppose that $\Lambda $ is aperiodic. We
trivially have $\mathcal{G}_{\Lambda }^{\left( 0\right) }$ belongs to the
interior of $\operatorname{Iso}\left( \mathcal{G}_{\Lambda }\right) $. Now we show
the reverse inclusion. Take $a$ an interior point of $\operatorname{Iso}\left(
\mathcal{G}_{\Lambda }\right) $. Then there exits $Z_{\Lambda }\left( \left.
\lambda \ast _{s}\mu \right\backslash G\right) $ such that $Z_{\Lambda
}\left( \left. \lambda \ast _{s}\mu \right\backslash G\right) \subseteq
\operatorname{Iso}\left( \mathcal{G}_{\Lambda }\right) $ and $a\in Z_{\Lambda
}\left( \left. \lambda \ast _{s}\mu \right\backslash G\right) $. We show $%
\lambda =\mu $.

Note that since $a\in Z_{\Lambda }\left( \left. \lambda \ast _{s}\mu
\right\backslash G\right) $, then $Z_{\Lambda }\left( \left. \lambda \ast
_{s}\mu \right\backslash G\right) $ is not empty and by Remark \ref%
{remark-of-groupouid-Glambda}.(ii), $G$ is not exhaustive. Hence, there
exists $\nu \in s\left( \lambda \right) \Lambda $ such that $\Lambda ^{\min
}\left( \nu ,\gamma \right) =\emptyset $ for $\gamma \in G$. Because $%
\Lambda $ is aperiodic, there exists a aperiodic boundary path $x\in s\left(
\nu \right) \partial \Lambda .$

We claim that the boundary path $\nu x$ is also aperiodic. Suppose for
contradiction that there exists $\lambda ^{\prime },\mu ^{\prime }\in
\Lambda r\left( \nu x\right) $ such that $\lambda ^{\prime }\neq \mu
^{\prime }$ and
\begin{equation}
\lambda ^{\prime }\left( \nu x\right) =\mu ^{\prime }\left( \nu x\right) .
\label{equ-lambdanux-equal-munux}
\end{equation}%
Since $\lambda ^{\prime },\mu ^{\prime },\nu \in \Lambda $, by unique the
factorisation property we have $\lambda ^{\prime }\neq \mu ^{\prime }$
implies $\lambda ^{\prime }\nu \neq \mu ^{\prime }\nu $. Now because $x$ is
aperiodic, $\lambda ^{\prime }\nu \neq \mu ^{\prime }\nu $ implies $\lambda
^{\prime }\nu \left( x\right) \neq \mu ^{\prime }\nu \left( x\right) $,
which contradicts \eqref{equ-lambdanux-equal-munux}. Hence, $\nu x$ is
aperiodic, as claimed.

Since $\lambda \nu x\in \left. Z_{\Lambda }\left( \lambda \right)
\right\backslash Z_{\Lambda }\left( \lambda \gamma \right) $ and $\mu \nu
x\in \left. Z_{\Lambda }\left( \mu \right) \right\backslash Z_{\Lambda
}\left( \mu \gamma \right) $ for $\gamma \in G$, we have
\begin{equation*}
\left( \lambda \nu x,d\left( \lambda \right) -d\left( \mu \right) ,\mu \nu
x\right) \in Z_{\Lambda }\left( \left. \lambda \ast _{s}\mu \right\backslash
G\right) \text{.}
\end{equation*}%
Thus $Z_{\Lambda }\left( \left. \lambda \ast _{s}\mu \right\backslash
G\right) \subseteq \operatorname{Iso}\left( \mathcal{G}_{\Lambda }\right) $, and
hence $\lambda \nu x=\mu \nu x$. Since $\nu x$ is aperiodic, we have $%
\lambda \left( \nu x\right) =\mu \left( \nu x\right) $ which implies $%
\lambda =\mu $. Therefore, $\mathcal{G}_{\Lambda }$ is effective.

$\left( \Leftarrow \right) $ Now suppose that $\Lambda $ is not aperiodic.
Then there exists $v\in \Lambda ^{0}$ such that for all boundary path $x\in
v\partial \Lambda $, $x$ is not aperiodic.

\begin{claim}
\label{xGx-not-equal-x}For $x\in v\partial \Lambda $, we have $x\mathcal{G}%
_{\Lambda }x\neq \left\{ x\right\} $.
\end{claim}

\begin{proof}[Proof of Claim~\protect\ref{xGx-not-equal-x}]
Take $x\in v\partial \Lambda $. Since $x$ is not aperiodic, then there exist
$\lambda ,\mu \in \Lambda r\left( x\right) $ such that $\lambda \neq \mu $
and $\lambda x=\mu x$. If $d\left( \lambda \right) =d\left( \mu \right) $,
then
\begin{equation*}
\lambda =\left( \lambda x\right) \left( 0,d\left( \lambda \right) \right)
=\left( \mu x\right) \left( 0,d\left( \mu \right) \right) =\mu \text{,}
\end{equation*}%
which contradicts with $\lambda \neq \mu $.

So suppose $d\left( \lambda \right) \neq d\left( \mu \right) $. Note that
for $1\leq i\leq k$ such that $d\left( \lambda \right) _{i}\neq d\left( \mu
\right) _{i}$, we have $d\left( x\right) _{i}=\infty $ (since $\lambda x=\mu
x$). Hence%
\begin{equation*}
\left( \left( d\left( \lambda \right) \vee d\left( \mu \right) \right)
-d\left( \lambda \right) \right) \vee \left( \left( d\left( \lambda \right)
\vee d\left( \mu \right) \right) -d\left( \mu \right) \right) \leq d\left(
x\right) \text{.}
\end{equation*}%
Write $p:=\left( d\left( \lambda \right) \vee d\left( \mu \right) \right)
-d\left( \lambda \right) $ and $q:=\left( d\left( \lambda \right) \vee
d\left( \mu \right) \right) -d\left( \mu \right) $. Then
\begin{align*}
\sigma ^{p}x &=\sigma ^{p}\left( \sigma ^{d\left( \lambda \right) }\left(
\lambda x\right) \right) =\sigma ^{d\left( \lambda \right) \vee d\left( \mu
\right) }\left( \lambda x\right) \\
&=\sigma ^{d\left( \lambda \right) \vee d\left( \mu \right) }\left( \mu
x\right) \text{ (since }\lambda x=\mu x\text{)} \\
&=\sigma ^{q}\left( \sigma ^{d\left( \mu \right) }\left( \mu x\right)
\right) =\sigma ^{q}x
\end{align*}%
and $p\neq q$ (since $d\left( \lambda \right) \neq d\left( \mu \right) $).
This implies $\left( x,p-q,x\right) \in \left. \mathcal{G}_{\Lambda
}\right\backslash \mathcal{G}_{\Lambda }^{\left( 0\right) }$ and $x\mathcal{G%
}_{\Lambda }x\neq \left\{ x\right\} $. \hfil\penalty100\hbox{}\nobreak\hfill
\hbox{\qed\
Claim~\ref{xGx-not-equal-x}} \renewcommand\qed{}
\end{proof}

Since $x\mathcal{G}_{\Lambda }x\neq \left\{ x\right\} $ for all $x\in
v\partial \Lambda $, then
\begin{equation*}
Z_{\Lambda }\left( v\right) \cap \{z\in \mathcal{G}_{\Lambda }^{\left(
0\right) }:z\mathcal{G}_{\Lambda }z=\left\{ z\right\} \}=\emptyset
\end{equation*}%
and $\{z\in \mathcal{G}_{\Lambda }^{\left( 0\right) }:z\mathcal{G}_{\Lambda
}z=\left\{ z\right\} \}$ is not dense in $\mathcal{G}_{\Lambda }^{\left(
0\right) }$. Since $\mathcal{G}_{\Lambda }$ is locally compact,
second-countable, Hausdorff and \'{e}tale, then by \cite[Proposition 3.6.(b)]%
{R08}, $\mathcal{G}_{\Lambda }$ is not effective, as required.
\end{proof}

\begin{remark}
In fact, for a finitely aligned $k$-graph $\Lambda $, the following five
conditions are equivalent:

\begin{enumerate}
\item[(a)] $\mathcal{G}_{\Lambda }$ is effective.

\item[(b)] $\mathcal{G}_{\Lambda }$ is \emph{topologically principal } in
that the set of units with trivial isotropy is dense in $\mathcal{G}^{(0)}$.

\item[(c)] $\mathcal{G}_{\Lambda }$ satisfies Condition (1) of Theorem 5.1
of \cite{RSWY12}.

\item[(d)] $\Lambda $ has \emph{no local periodicity} as defined in \cite%
{Sh12}.

\item[(e)] $\Lambda $ is aperiodic.
\end{enumerate}

In \cite[Proposition 3.6]{R08}, Renault shows that for a locally compact,
second-countable, Hausdorff, \'{e}tale $\mathcal{G}$, $\mathcal{G}$ is
effective if and only if it is topologically principle. Since the
boundary-path groupoid $\mathcal{G}_{\Lambda }$ is locally compact,
second-countable, Hausdorff and \'{e}tale, then (a)$\Leftrightarrow $(b).
Meanwhile, in \cite[Theorem 5.2]{Y07}, Yeend proves (b)$\Leftrightarrow $%
(c). [Note that Yeend uses notion \textquotedblleft \emph{essentially free}%
\textquotedblright\ instead of \textquotedblleft topologically
principal\textquotedblright .] Lemma 5.6 of \cite{RSWY12} gives (c)$%
\Leftrightarrow $(d). Finally, (d)$\Leftrightarrow $(e) follows from \cite[%
Proposition 2.11]{Sh12}.
\end{remark}

\section{Cofinal higher-rank graphs and minimal groupoids}

\label{Section-cofinal-minimal} In this section, we show that a finitely
aligned $k$-graph $\Lambda $ is cofinal if and only if the boundary-path
groupoid $\mathcal{G}_{\Lambda }$ is minimal (Proposition \ref%
{cofinal-iff-minimal}). Later, we use this relationship to study the
simplicity of Kumjian-Pask algebras in Section \ref%
{Section-Basic-Simpllicity}.

Recall from \cite[Definition 8.4]{Si06(G)} that we say a $k$-graph $\Lambda $
is \emph{cofinal} if for all $v\in \Lambda ^{0}$ and $x\in \partial \Lambda $%
, there exists $n\leq d\left( x\right) $ such that $v\Lambda x\left(
n\right) \neq \emptyset $.

In a groupoid $\mathcal{G}$, a subset $U\subseteq \mathcal{G}^{\left(
0\right) }$ is called \emph{invariant}\ if $s\left( a\right) \in U$ implies $%
r\left( a\right) \in U$ for all $a\in \mathcal{G}$. Note that $U$ is
invariant if and only if $\mathcal{G}^{\left( 0\right) }\backslash U$ is
invariant. We then say a topological groupoid $\mathcal{G}$ is \emph{minimal}
if $\mathcal{G}^{\left( 0\right) }$ has no nontrivial open invariant
subsets. Equivalently, $\mathcal{G}$ is minimal if for each $x\in \mathcal{G}%
^{\left( 0\right) }$, the orbit $\left[ x\right] :=s\left( x\mathcal{G}%
\right) $ is dense in $\mathcal{G}^{\left( 0\right) }$.

\begin{proposition}
\label{cofinal-iff-minimal}Let $\Lambda $ be a finitely aligned $k$-graph.
Then $\Lambda $ is cofinal if and only if the boundary-path groupoid $%
\mathcal{G}_{\Lambda }$ is minimal.
\end{proposition}

\begin{proof}
$\left( \Rightarrow \right) $ Suppose that $\Lambda $ is cofinal. Take $x\in
\mathcal{G}_{\Lambda }^{\left( 0\right) }$. We have to show that $\left[ x%
\right] $ is dense in $\mathcal{G}_{\Lambda }^{\left( 0\right) }$. Take a
nonempty open set $Z_{\Lambda }\left( \left. \lambda \right\backslash
G\right) $ and we claim that $Z_{\Lambda }\left( \left. \lambda
\right\backslash G\right) \cap \left[ x\right] \neq \emptyset $. Since $%
Z_{\Lambda }\left( \left. \lambda \right\backslash G\right) $ is nonempty,
we have that $G$ is not exhaustive (see Remark \ref%
{remark-of-groupouid-Glambda}.(i)). Then there exists $\nu \in s\left(
\lambda \right) \Lambda $ such that $\Lambda ^{\min }\left( \nu ,\gamma
\right) =\emptyset $ for $\gamma \in G$. Now consider the vertex $s\left(
\lambda \nu \right) $ and the boundary path $x$. Since $\Lambda $ is
cofinal, then there exists $n\leq d\left( x\right) $ such that $s\left(
\lambda \nu \right) \Lambda x\left( n\right) \neq \emptyset $. Take $\mu \in
s\left( \lambda \nu \right) \Lambda x\left( n\right) $. Because $x$ is a
boundary path, so is $\sigma ^{n}x$. Hence,
\begin{equation*}
y:=\lambda \nu \mu \left[ \sigma ^{n}x\right]
\end{equation*}%
is also a boundary path. It is clear that $y\in Z_{\Lambda }\left( \lambda
\right) $ and since $\Lambda ^{\min }\left( \nu ,\gamma \right) =\emptyset $
for $\gamma \in G$, we have $y\notin Z_{\Lambda }\left( \lambda \gamma
\right) $ for $\gamma \in G$. Hence, $y\in Z_{\Lambda }\left( \left. \lambda
\right\backslash G\right) $.

On the other hand, since $y=\lambda \nu \mu \left[ \sigma ^{n}x\right] $,
then $\left( x,n-d\left( \lambda \nu \mu \right) ,y\right) \in \mathcal{G}%
_{\Lambda }$ and $y\in \left[ x\right] $. Therefore, $Z_{\Lambda }\left(
\left. \lambda \right\backslash G\right) \cap \left[ x\right] \neq \emptyset
$. Thus, $\left[ x\right] $ is dense in $\mathcal{G}_{\Lambda }^{\left(
0\right) }$ and $\mathcal{G}_{\Lambda }$ is minimal.

$\left( \Leftarrow \right) $ Suppose that $\Lambda $ is not cofinal. Then
there exist $v\in \Lambda ^{0}$ and $x\in \partial \Lambda $ such that for
all $n\leq d\left( x\right) ,$ we have $v\Lambda x\left( n\right) =\emptyset
$. We claim $Z_{\Lambda }\left( v\right) \cap \left[ x\right] =\emptyset $.
Suppose for contradiction that $Z_{\Lambda }\left( v\right) \cap \left[ x%
\right] \neq \emptyset $. Take $y\in Z_{\Lambda }\left( v\right) \cap \left[
x\right] $. Because $y\in \left[ x\right] $, then there exists $p,q\in
\mathbb{N}^{k}$ such that $\left( x,p-q,y\right) \in \mathcal{G}_{\Lambda }$%
. This implies $\sigma ^{p}x=\sigma ^{q}y$. Since $y\in Z_{\Lambda }\left(
v\right) $, then $r\left( y\right) =v$. Hence, $\sigma ^{p}x=\sigma ^{q}y$
and $r\left( y\right) =v$ imply that $y\left( 0,q\right) $ belongs to $%
v\Lambda x\left( p\right) $, which contradicts with $v\Lambda x\left(
n\right) =\emptyset $ for all $n\leq d\left( x\right) $. Therefore, $%
Z_{\Lambda }\left( v\right) \cap \left[ x\right] =\emptyset $, as claimed,
and $\left[ x\right] $ is not dense in $\mathcal{G}_{\Lambda }^{\left(
0\right) }$. Thus, $\mathcal{G}_{\Lambda }$ is not minimal.
\end{proof}

\section{The Cuntz-Krieger uniqueness theorem}

\label{Section-the-CK-uniqueness-theorem}Throughout this section, $\Lambda $
is a finitely aligned $k$-graph and $R$ is a commutative ring with identity $%
1$.

\begin{theorem}[The Cuntz-Krieger uniqueness theorem]
\label{the-CK-uniqueness-theorem}Let $\Lambda $ be an aperiodic finitely
aligned $k$-graph, $R$ be a commutative ring with $1$. Suppose that $\pi :%
{\normalsize \operatorname{KP}}_{R}\left( \Lambda \right) \rightarrow A$ is an $R$%
-algebra homomorphism such that $\pi \left( rs_{v}\right) \neq 0$ for all $%
r\in \left. R\right\backslash \left\{ 0\right\} $ and $v\in \Lambda ^{0}$.
Then $\pi $ is injective.
\end{theorem}

We show Theorem \ref{the-CK-uniqueness-theorem} by using the Cuntz-Krieger
uniqueness theorem for Steinberg algebras \cite[Theorem 3.2]{CE-M15}. First
we verify an alternate formulation of the Cuntz-Krieger uniqueness theorem
for Steinberg algebras that will be useful.

\begin{theorem}
\label{the-CK-uniqueness-theorem-for-Steinberg-algebras}Let $\mathcal{G}$ be
an effective, Hausdorff, ample groupoid, and $R$ be a commutative ring with $%
1$. Let $\mathcal{B}$ be a basis of compact open bisection for the topology
on $\mathcal{G}$. Let $\phi :A_{R}\left( \mathcal{G}\right) \rightarrow A$
be an $R$-algebra homomorphism. Suppose that $\ker \left( \phi \right) \neq
0 $. Then there exist $r\in \left. R\right\backslash \left\{ 0\right\} $ and
$B\in \mathcal{B}$ such that $B\subseteq \mathcal{G}^{\left( 0\right) }$ and
$\phi \left( r1_{B}\right) =0$.
\end{theorem}

\begin{proof}
Since $\ker \left( \phi \right) \neq 0$, then by \cite[Theorem 3.2]{CE-M15},
there exist $r\in \left. R\right\backslash \left\{ 0\right\} $ and a
nonempty compact open subset $K\subseteq \mathcal{G}^{\left( 0\right) }$
such that $\phi \left( r1_{K}\right) =0$. Since $K$ is open, then there is $%
B\in \mathcal{B}$ such that $B\subseteq K$. Hence, $B\subseteq \mathcal{G}%
^{\left( 0\right) }$ and%
\begin{equation*}
0=\phi \left( r1_{K}\right) \phi \left( 1_{B}\right) =\phi \left(
r1_{KB}\right) =\phi \left( r1_{K\cap B}\right) =\phi \left( r1_{B}\right)
\text{.}
\end{equation*}
\end{proof}

\begin{proof}[Proof of Theorem \protect\ref{the-CK-uniqueness-theorem}]
First note that $\mathcal{G}_{\Lambda }$ is a Hausdorff and ample groupoid
that is effective by Proposition~\ref{aperiodic-iff-effective}. Thus it
satisfies the hypothesis of Theorem \ref%
{the-CK-uniqueness-theorem-for-Steinberg-algebras}. Now recall the
isomorphism $\pi _{T}:{\normalsize \operatorname{KP}}_{R}\left( \Lambda \right)
\rightarrow A_{R}\left( \mathcal{G}_{\Lambda }\right) $ as in Proposition %
\ref{KP-is-isomorphic-to-Steinberg-algebras}. Then $\pi _{T}\left(
s_{\lambda }\right) =1_{Z_{\Lambda }\left( \lambda \ast _{s}s\left( \lambda
\right) \right) }$ and $\pi _{T}\left( s_{\mu ^{\ast }}\right)
=1_{Z_{\Lambda }\left( s\left( \mu \right) \ast _{s}\mu \right) }$ for $%
\lambda ,u\in \Lambda $. Define $\phi :=\pi \circ \pi _{T}^{-1}$. To show
the injectivity of $\pi $, it suffices to show that $\phi $ is injective.
Suppose for contradiction that $\phi $ is not injective. By Theorem \ref%
{the-CK-uniqueness-theorem-for-Steinberg-algebras}, there exist $r\in \left.
R\right\backslash \left\{ 0\right\} $ and $Z_{\Lambda }\left( \left. \lambda
\right\backslash G\right) $ such that $\phi \left( r1_{Z_{\Lambda }\left(
\left. \lambda \right\backslash G\right) }\right) =0$. Since $1_{Z_{\Lambda
}\left( \left. \lambda \right\backslash G\right) }$ can be identified as $%
1_{Z_{\Lambda }\left( \left. \lambda \ast _{s}\lambda \right\backslash
G\right) }$ (Remark \ref{remark-of-groupouid-Glambda}.(i)), then by
following the argument of \eqref{equ-1_Z(lambda*mu-G)}, we get%
\begin{equation*}
\phi \left( r1_{Z_{\Lambda }\left( \left. \lambda \right\backslash G\right)
}\right) =\pi \big(rs_{\lambda }\big(\prod_{\nu \in G}\left( s_{s\left(
\lambda \right) }-s_{\nu }s_{\nu ^{\ast }}\right) \big)s_{\lambda ^{\ast }}%
\big)
\end{equation*}%
and then%
\begin{equation}
\pi \big(rs_{\lambda }\big(\prod_{\nu \in G}\left( s_{s\left( \lambda
\right) }-s_{\nu }s_{\nu ^{\ast }}\right) \big)s_{\lambda ^{\ast }}\big)=0%
\text{.}  \label{equ-equal-0-the-CK-uniqueness-thm}
\end{equation}

On the other hand, since $\pi \left( rs_{v}\right) \neq 0$ for all $r\in
\left. R\right\backslash \left\{ 0\right\} $ and $v\in \Lambda ^{0}$, and $G$
is finite non-exhaustive, then by Proposition \ref{properties-of-KP}.(d),
\begin{equation*}
\pi \big(rs_{\lambda }\big(\prod_{\nu \in G}\left( s_{s\left( \lambda
\right) }-s_{\nu }s_{\nu ^{\ast }}\right) \big)s_{\lambda ^{\ast }}\big)\neq
0\text{,}
\end{equation*}%
which contradicts \eqref{equ-equal-0-the-CK-uniqueness-thm}. The conclusion
follows.
\end{proof}

One application of Theorem \ref{the-CK-uniqueness-theorem} is:

\begin{corollary}
\label{injectivity-of-the-boundary-path-representation}Let $\Lambda $ be
finitely aligned $k$-graph and $R$ be a commutative ring with $1$. Then $%
\Lambda $ is aperiodic if and only if the boundary-path representation $\pi
_{S}:{\normalsize \operatorname{KP}}_{R}\left( \Lambda \right) \rightarrow \operatorname{End}%
\left( \mathbb{F}_{R}\left( \partial \Lambda \right) \right) $ is injective.
\end{corollary}

To show Corollary \ref{injectivity-of-the-boundary-path-representation}, we
establish some results and notation.

Following \cite[Definition 2.3]{Sh12}, for a finitely aligned $k$-graph $%
\Lambda $, we say $\Lambda $ has \emph{no local periodicity}\ if for every $%
v\in \Lambda ^{0}$ and every $n\neq m\in \mathbb{N}^{k}$, there exists $x\in
v\partial \Lambda $ such that either $d\left( x\right) \ngeq n\vee m$ or $%
\sigma ^{n}x\neq \sigma ^{m}x$. If no local aperiodicity fails at $v\in
\Lambda ^{0}$, then there are $n\neq m\in \mathbb{N}^{k}$ such that $\sigma
^{n}x=\sigma ^{m}x$ for all $x\in v\partial \Lambda $. In this case, we say $%
\Lambda $ has \emph{local periodicity}\ $n,m$ \emph{at }$v\in \Lambda ^{0}$.

\begin{lemma}[{\protect\cite[Lemma 2.9]{Sh12}}]
\label{lemma-has-local-periodicity}Let $\Lambda $ be a finitely aligned $k$%
-graph which has local periodicity $n,m$ at $v\in \Lambda ^{0}$. Then $%
d\left( x\right) \geq n\vee m$ and $\sigma ^{n}x=\sigma ^{m}x$ for every $%
x\in v\partial \Lambda $. Fix $x\in v\partial \Lambda $ and set $\mu
=x\left( 0,m\right) $, $\alpha =x\left( m,m\vee n\right) $, and $\nu =\left(
\mu \alpha \right) \left( 0,n\right) $. Then $\mu \alpha y=\nu \alpha y$ for
every $y\in s\left( \alpha \right) \partial \Lambda $.
\end{lemma}

\begin{proof}[Proof of Corollary \protect\ref%
{injectivity-of-the-boundary-path-representation}]
$\left( \Rightarrow \right) $ Suppose that $\Lambda $ is aperiodic. By
Proposition \ref{the-boundary-path-representation}, we have $\pi _{S}\left(
rs_{v}\right) \neq 0$ for all $r\in \left. R\right\backslash \left\{
0\right\} $ and $v\in \Lambda ^{0}$. Since $\Lambda $ is aperiodic, then by
Theorem \ref{the-CK-uniqueness-theorem}, $\pi _{S}$ is injective.

$\left( \Leftarrow \right) $ Suppose that $\Lambda $ is not aperiodic. We
are following the argument of \cite[Lemma 5.9]{ACaHR13}. Since $\Lambda $ is
not aperiodic, by \cite[Proposition 2.11]{Sh12}, there exist $v\in \Lambda
^{0}$ and $n\neq m\in \mathbb{N}^{k}$ such that $\Lambda $ has local
periodicity $n,m$ at $v\in \Lambda ^{0}$. Let $\mu ,\nu ,\alpha $ be as in
Lemma \ref{lemma-has-local-periodicity} and define $a:=s_{\mu \alpha
}s_{\left( \mu \alpha \right) ^{\ast }}-s_{\nu \alpha }s_{\left( \mu \alpha
\right) ^{\ast }}$. We claim that $a\in \left. \ker \left( \pi _{S}\right)
\right\backslash \left\{ 0\right\} $.

First we show that $a\neq 0$. Suppose for contradiction that $a=0$. Then $%
s_{\mu \alpha }s_{\left( \mu \alpha \right) ^{\ast }}=s_{\nu \alpha
}s_{\left( \mu \alpha \right) ^{\ast }}$. Note that $d\left( s_{\mu \alpha
}s_{\left( \mu \alpha \right) ^{\ast }}\right) =d\left( \mu \alpha \right)
-d\left( \mu \alpha \right) =0$ and%
\begin{equation*}
d\left( s_{\nu \alpha }s_{\left( \mu \alpha \right) ^{\ast }}\right)
=d\left( \nu \alpha \right) -d\left( \mu \alpha \right) =d\left( \nu \right)
+d\left( \alpha \right) -d\left( \mu \right) -d\left( \alpha \right)
=n-m\neq 0\text{.}
\end{equation*}%
Hence $s_{\mu \alpha }s_{\left( \mu \alpha \right) ^{\ast }}=s_{\nu \alpha
}s_{\left( \mu \alpha \right) ^{\ast }}=0$. Thus, $0=s_{\left( \mu \alpha
\right) ^{\ast }}\left( s_{\mu \alpha }s_{\left( \mu \alpha \right) ^{\ast
}}\right) s_{\mu \alpha }=s_{s\left( \mu \alpha \right) }^{2}=s_{s\left( \mu
\alpha \right) }$, which contradicts Theorem \ref{universal-KP-family}.(b).
Hence $a\neq 0$.

Now we show that $a\in \ker \left( \pi _{S}\right) $. Take $y\in \partial
\Lambda $, and it suffices to show $\pi _{S}\left( a\right) \left( y\right)
=0$. Recall that $\pi _{S}\left( s_{\lambda }\right) =S_{\lambda }$ and $\pi
_{S}\left( s_{\mu ^{\ast }}\right) =S_{\mu ^{\ast }}$ where
\begin{equation*}
S_{\lambda }\left( y\right) =%
\begin{cases}
\lambda y & \text{if }s\left( \lambda \right) =r\left( y\right) \text{;} \\
0 & \text{otherwise,}%
\end{cases}%
\text{ and }S_{\mu ^{\ast }}\left( y\right) =%
\begin{cases}
\sigma ^{d\left( \mu \right) }y & \text{if }y\left( 0,d\left( \mu \right)
\right) =\mu \text{;} \\
0 & \text{otherwise.}%
\end{cases}%
\end{equation*}%
First suppose that $y\left( 0,d\left( \mu \alpha \right) \right) \neq \mu
\alpha $. Then $S_{\left( \mu \alpha \right) ^{\ast }}\left( y\right) =0$
and $\pi _{S}\left( a\right) \left( y\right) =S_{\mu \alpha }S_{\left( \mu
\alpha \right) ^{\ast }}\left( y\right) -S_{\nu \alpha }S_{\left( \mu \alpha
\right) ^{\ast }}\left( y\right) =0$. Next suppose that $y\left( 0,d\left(
\mu \alpha \right) \right) =\mu \alpha $. Then
\begin{equation*}
\pi _{S}\left( a\right) \left( y\right) =\left( S_{\mu \alpha }-S_{\nu
\alpha }\right) \left( \sigma ^{d\left( \mu \alpha \right) }y\right) \text{.}
\end{equation*}%
Since $y\in \partial \Lambda $, then $\sigma ^{d\left( \mu \alpha \right)
}y\in s\left( \alpha \right) \partial \Lambda $ and by Lemma \ref%
{lemma-has-local-periodicity}, $\mu \alpha \left( \sigma ^{d\left( \mu
\alpha \right) }y\right) =\nu \alpha \left( \sigma ^{d\left( \mu \alpha
\right) }y\right) $ and hence $\pi _{S}\left( a\right) \left( y\right) =0$.
Thus, $a\in \left. \ker \left( \pi _{S}\right) \right\backslash \left\{
0\right\} $, as claimed, and $\pi _{S}$ is not injective.
\end{proof}

\section{Basic simplicity and simplicity}

\label{Section-Basic-Simpllicity}As in \cite{T11}, we say an ideal $I$ in $%
{\normalsize \operatorname{KP}}_{R}\left( \Lambda \right) $ is \emph{basic} if
whenever $r\in \left. R\right\backslash \left\{ 0\right\} $ and $v\in
\Lambda ^{0}$, we have $rs_{v}\in I$ implies $s_{v}\in I$. We also say that $%
{\normalsize \operatorname{KP}}_{R}\left( \Lambda \right) $ is \emph{basically simple%
}\ if its only basic ideals are $\left\{ 0\right\} $ and ${\normalsize \operatorname{%
KP}}_{R}\left( \Lambda \right) $.

In this section, we investigate necessary and sufficient conditions for $%
{\normalsize \operatorname{KP}}_{R}\left( \Lambda \right) $ to be basically simple
(Theorem \ref{basic-simplicity}) and to be simple (Theorem \ref{simplicity}%
). We show that both results can be viewed as a consequences of basic
simplicity and simplicity characterisations of Steinberg algebras.
Therefore, we state necessary and sufficient conditions for the Steinberg
algebra $A_{R}\left( \mathcal{G}\right) $ to be basically simple and to be
simple in the following two theorems.

\begin{theorem}[{\protect\cite[Theorem 4.1]{CE-M15}}]
\label{basic-simplicity-for-Steinberg-algebras}Let $\mathcal{G}$ be an
Hausdorff, ample groupoid and $R$ be a commutative ring with $1$. Then $%
A_{R}\left( \mathcal{G}\right) $ is basically simple if and only if $%
\mathcal{G}$ is effective and minimal.
\end{theorem}

\begin{theorem}[{\protect\cite[Corollary 4.6]{CE-M15}}]
\label{simplicity-for-Steinberg-algebras}Let $\mathcal{G}$ be an Hausdorff,
ample groupoid and $R$ be a commutative ring with $1$. Then $A_{R}\left(
\mathcal{G}\right) $ is simple if and only if $R$ is a field and $\mathcal{G}
$ is effective and minimal.
\end{theorem}

Now we are ready to prove our results in this section.

\begin{theorem}
\label{basic-simplicity}Let $\Lambda $ be a finitely aligned $k$-graph and
let $R$ be a commutative ring with $1$. Then ${\normalsize \operatorname{KP}}%
_{R}\left( \Lambda \right) $ is basically simple if and only if $\Lambda $
is aperiodic and cofinal.
\end{theorem}

\begin{proof}
$\left( \Rightarrow \right) $ First suppose that ${\normalsize \operatorname{KP}}%
_{R}\left( \Lambda \right) $ is basically simple. By Proposition \ref%
{KP-is-isomorphic-to-Steinberg-algebras}, $A_{R}\left( \mathcal{G}_{\Lambda
}\right) $ is also basically simple and then by Theorem \ref%
{basic-simplicity-for-Steinberg-algebras}, $\mathcal{G}_{\Lambda }$ is
effective and minimal. On the other hand, $\mathcal{G}_{\Lambda }$ is
effective implies that $\Lambda $ is aperiodic (Proposition \ref%
{aperiodic-iff-effective}), and $\mathcal{G}_{\Lambda }$ is minimal implies
that $\Lambda $ is cofinal (Proposition \ref{cofinal-iff-minimal}). The
conclusion follows.

$\left( \Leftarrow \right) $ Next suppose that $\Lambda $ is aperiodic and
cofinal. By Proposition \ref{aperiodic-iff-effective} and Proposition \ref%
{cofinal-iff-minimal}, $\mathcal{G}_{\Lambda }$ is effective and minimal and
then by Theorem \ref{basic-simplicity-for-Steinberg-algebras}, $A_{R}\left(
\mathcal{G}_{\Lambda }\right) $ is basically simple. Since $A_{R}\left(
\mathcal{G}_{\Lambda }\right) $ is isomorphic to ${\normalsize \operatorname{KP}}%
_{R}\left( \Lambda \right) $ (Proposition \ref%
{KP-is-isomorphic-to-Steinberg-algebras}), then ${\normalsize \operatorname{KP}}%
_{R}\left( \Lambda \right) $ is also basically simple, as required.
\end{proof}

\begin{theorem}
\label{simplicity}Let $\Lambda $ be a finitely aligned $k$-graph and let $R$
be a commutative ring with $1$. Then ${\normalsize \operatorname{KP}}_{R}\left(
\Lambda \right) $ is simple if and only if $R$ is a field and $\Lambda $ is
aperiodic and cofinal.
\end{theorem}

\begin{proof}
$\left( \Rightarrow \right) $ First suppose that ${\normalsize \operatorname{KP}}%
_{R}\left( \Lambda \right) $ is simple. Then ${\normalsize \operatorname{KP}}%
_{R}\left( \Lambda \right) $ is also basically simple and Theorem \ref%
{basic-simplicity} implies that $\Lambda $ is aperiodic and cofinal. On the
other hand, since ${\normalsize \operatorname{KP}}_{R}\left( \Lambda \right) $ is
simple, then by Proposition \ref{KP-is-isomorphic-to-Steinberg-algebras}, $%
A_{R}\left( \mathcal{G}_{\Lambda }\right) $ is also simple and by Theorem %
\ref{simplicity-for-Steinberg-algebras}, $R$ is a field, as required.

$\left( \Leftarrow \right) $ Next suppose that $R$ is a field and $\Lambda $
is aperiodic and cofinal. By Proposition \ref{aperiodic-iff-effective} and
Proposition \ref{cofinal-iff-minimal}, $\mathcal{G}_{\Lambda }$ is effective
and minimal. Hence, by Theorem \ref{simplicity-for-Steinberg-algebras}, $%
A_{R}\left( \mathcal{G}_{\Lambda }\right) $ is simple and by Proposition \ref%
{KP-is-isomorphic-to-Steinberg-algebras}, so is ${\normalsize \operatorname{KP}}%
_{R}\left( \Lambda \right) $.
\end{proof}

\end{document}